\documentclass[11pt,reqno]{amsart}
\usepackage{geometry}
\usepackage{graphicx}
\usepackage{amssymb}
\usepackage{pgfplots}
\usepackage{amsmath, mathrsfs, stmaryrd, color,slashed}
\newtheorem{theorem}{Theorem}
\usepackage{color}
\usepackage[title]{appendix}

\newtheorem{definition}[theorem]{Definition}
\newtheorem{proposition}[theorem]{Proposition}
\newtheorem{corollary}[theorem]{Corollary}
\newtheorem{lemma}[theorem]{Lemma}
\newcommand{\R}{\mathbb R}
\newcommand{\D}{\left(1-\frac{2M}{r}\right)}
\newcommand{\nablas}{\slashed{\nabla}}
\newcommand{\nablab}{\nabla}
\newcommand{\Div}{\textup{div}}

\numberwithin{theorem}{section}
\numberwithin{equation}{section}

\begin{document} 

\title[linear stability of Schwarzschild in harmonic gauge: odd part]{The linear stability of the Schwarzschild spacetime in the harmonic gauge: odd part}

\author{Pei-Ken Hung}
\address{Pei-Ken Hung\\
Department of Mathematics\\
Columbia University, USA}
\email{pkhung@math.columbia.edu}

\begin{abstract}
In this paper, we study the odd solution of the linearlized Einstein equation on the Schwarzschild background and in the harmonic gauge. With the aid of Regge-Wheeler quantities, we are able to estimate the odd part of Lichnerowicz d'Alembertian equation. In particular, we prove the solution decays at rate $\tau^{-1+\delta}$ to a linearlized Kerr solution.
\end{abstract}
\maketitle

\section{Introduction}
One major open problem of mathematical general relativity is the stability of Kerr spacetimes as solutions of the vacuum Einstein equation
\begin{align}\label{VEE}
Ric(g)=0.
\end{align}
Equation (\ref{VEE}) is a quasilinear weakly hyperbolic system of the metric $g$. The Kerr family is conjectured to be stable for it to be physically relevant. However, in very few special cases the non-linear stability is mathematically understood. The non-linear stability of the Minkowski spacetime was proved in the monumental work of Christodoulou and Klainerman \cite{Christodoulou-Klainerman}. See also \cite{Klainerman-Nicolo, Bieri, Lindblad-Rodnianski, Hintz-Vasy} for various approaches. Recently, non-linear stability of the Schwarzschild spacetime was established by Klainerman and Szeftel \cite{Klainerman-Szeftel} for axial symmetric polarized perturbations. However, general non-linear stability of the Schwarzschild spactime remains open.\\

A simplified model problem is to study the linearlization of equation (\ref{VEE}) on a specific background
\begin{equation}\label{LVEE}
\begin{split}
\delta Ric\Big|_{g}(h)=0,
 \end{split}
\end{equation}
where $g$ is a metric satisfying (\ref{VEE}). A solution $h$ of the equation (\ref{LVEE}) is referred to as metric perturbation and whether $h$ remains bounded or even decays is the problem of linear stability. Because the vacuum Einstein equation (\ref{VEE}) is invariant under the diffeomorphism group, its linearlization (\ref{LVEE}) has a infinite dimensional solution space consisting of deformation tensors. Therefore gauge-invariant quantities play a crucial role in both non-linear and linear stability.\\

The study of metric perturbation on the Schwarzschild background was initiated by Regge and Wheeler \cite{Regge-Wheeler}. The authors derived the Regge-Wheeler equation for the odd solution of (\ref{LVEE}). For even solutions, there is a similar equation discovered by Zerilli \cite{Zerilli}. Later, Moncrief \cite{Moncrief, Moncrief2} formulated the Regge-Wheeler and Zerilli quantities in a gauge-invariant way. Bardeen and Press \cite{Bardeen-Press} adapted the Newman-Penrose formalism to study equation (\ref{LVEE}). This approach was extended to Kerr spacetimes by Teukolsky \cite{Teukolsky}, showing that the extreme Weyl curvature components satisfy the Teukolsky equations. In the Schwarzschild spacetime, the transformation theory of Wald \cite{Wald} and Chandrasekhar \cite{Chandrasekhar} relates Regge-Wheeler-Zerilli-Moncrief system to the Teukolsky equations. See also \cite{Aksteiner-Andersson-Backdahl-Shah-Whiting} for further refinement of the transformation theory. These works accumulated to the proof of mode stability for Kerr by Whiting \cite{Whiting}.\\

A major breakthrough which goes beyond mode stability is the recent work of quantitative linear stability of Schwarzschild by Dafermos, Holzegel and Rodnianski \cite{Dafermos-Holzegel-Rodnianski}. This work can be roughly divided into three parts. First, the authors imposes the double null gauge on the metric perturbation $h$ and showed that the gauged equation is well-posed. Second, the vector field method developed by the first and the third author \cite{Dafermos-Rodnianski, Dafermos-Rodnianski2} is applied to estimate the Regge-Wheeler equation. The estimates for the Teukolsky equations are then derived via the transformation theory. Since the Regge-Wheeler and Teukolsky quantities are gauge-invariant, the estimates of this part in fact hold for all solutions of $(\ref{LVEE})$ without any gauge condition. Third, the estimates of the remaining gauge dependent quantities are derived by integrating the gauged equation together with elliptic estimates.\\

In this paper we study the odd solution of (\ref{LVEE}) under the harmonic gauge condition:
\begin{align}\label{HG}
\nabla^a\left(h_{ab}-\frac{1}{2}(\textup{tr}_g h)g_{ab}\right)=0.
\end{align}
The harmonic gauge is the linearlization of the harmonic map gauge used by Choquet-Bruhat \cite{Choquet-Bruhat} and Choquet-Bruhat-Geroch \cite{Choquet-Bruhat-Geroch} in proving short time existence and uniqueness of the vacuum Einstein equation (\ref{VEE}) and its maximal development. In the Riemannian setting, DeTurck \cite{DeTurk} applied the harmonic map gauge to prove the short time existence and the uniqueness of compact Ricci flow. Under the harmonic gauge, linearlized equation (\ref{LVEE}) is reduced to the Lichnerowicz d'Alembertian equation
\begin{align}\label{LVEEG}
\Box h_{ab}+2R_{acbd}h_{cd}=0.
\end{align}
From the standard theory of linear hyperbolic equations, the Cauchy problem for equation (\ref{LVEEG}) is well-posed and the solution exists for all time for regular initial data. Furthermore, the gauge condition (\ref{HG}) is preserved under (\ref{LVEEG}). For any solution $h$ of (\ref{LVEEG}), the gauge 1-from $\Gamma_b[h]$ satisfies the wave equation 
\begin{align*}
\Gamma_b[h]&:=\nabla^a\left(h_{ab}-\frac{1}{2}(\textup{tr}_g h)g_{ab}\right),\\
\Box \Gamma_b&=0.
\end{align*}
Therefore a solution of (\ref{LVEEG}) is also a solution of (\ref{LVEE}) provided $\Gamma_b$ and its normal derivative vanish initially. Conversely, any solution of (\ref{LVEE}) can be put into the harmonic gauge after subtracting a deformation tensor ${}^X\pi=\mathcal{L}_X g$. The gauge 1-form of ${}^X\pi$ is
\begin{align*}
\Gamma_b [{}^X\pi]=\Box X_b.
\end{align*}
Thus for any solution $h$ of (\ref{LVEE}), one can solve $\Box X_b=\Gamma_b [h]$ and then $\tilde{h}=h-{}^X\pi$ is a solution of (\ref{LVEEG}) and (\ref{HG}) simultaneously.\\

The main difference between this paper and \cite{Dafermos-Holzegel-Rodnianski} is the method to estimate gauge dependent quantities. Equation (\ref{LVEEG}) is a wave equation for the metric perturbation and we want to use the vector field method to control it directly. One nice property of (\ref{LVEEG}) is that it is reduced to ten linear scalar wave equations in the Minkowski spacetime. Therefore its behavior is well understood near the spatial infinity. Then we use the Regge-Wheeler quantities to control the error terms coming from nonzero mass $M>0$. Our proof requires further decomposition according to the angular mode. Recently, Johnson \cite{Johnson} uses a generalized harmonic gauge to study the linearlized gravity (\ref{LVEE}). In \cite{Johnson}, even/odd or angular mode decomposition is not needed and the decay estimate is obtained under the adapted Regge-Wheeler gauge.\\

Summarizing our result, we have the following linear stability theorem on the Schwarzschild spacetime for odd solutions:
\begin{theorem}[rough version]
Let $h$ be an odd solution of (\ref{LVEEG}) under the harmonic gauge condition (\ref{HG}). Define
\begin{align*}
\hat{h}:=h-\sum_{m=-1,0,1}d_m K_m,
\end{align*}
where $d_m$ are three numbers determined by $h$ and $K_m$ correspond to linearlized angular momentum solutions(see section \ref{ddd}). Further assume that $\hat{h}$ decays toward infinity fast enough on $\Sigma_0$ in the sense that some weighted norms are bounded. Then certain weighted $L^2$ norm and $L^\infty$ norm decay at rate $\tau^{-1+\delta}$ along the foliation $\Sigma_\tau$.
\end{theorem}
This paper is organized as follows. In section 2 we introduce the notations and at the end the precise statement of the main theorem. In section 3 and 4 we estimate the solutions with $\ell\geq 2$. In section 5 we discuss the case $\ell=1$ and prove the main theorem. Appendix contains some basic computations.
\section*{Acknowledgments} 
The author is grateful to Simon Brendle for suggesting this problem to him and for his continuous support. The author also thanks Sergiu Klainerman and Mu-Tao Wang for their encouragement. The author thanks T\"ubingen University where part of this work was carried out.
\newpage
\section{Preliminary}
\subsection{Schwarzschild spacetime}
The Schwarzschild spacetime with mass $M>0$ is the spherical symmetric, static Lorenzian manifold $(\mathcal{M},g_M)$ which satisfies the vacuum Einstein equation (\ref{VEE}). The Schwarzschild spacetime was discovered by Schwarzschild as the first non-trivial solution of (\ref{VEE}). The Schwarzschild exterior can be parametrized by a coordinate $(t,r,\theta,\phi)=(x^0,x^1,x^2,x^3)$ with $t\in \R$, $r>2M$ and $(\theta,\phi)\in S^2$. Under this coordinate system, the metric is of the form
\begin{align}
{g}_M=-\left(1-\frac{2M}{r} \right)dt^2+\left(1-\frac{2M}{r}\right)^{-1}dr^2+r^2(d\theta^2+\sin^2\theta d\phi^2).
\end{align}
We use the following indices notation: $a,b,c,\dots =0,1,2,3$ for the spacetime indices; $A,B,C,\dots =0,1$ for the quotient indices; $\alpha,\beta,\gamma\dots 2,3$ for the spherical indices. The quotient metric $\tilde{g}$ and the spherical metric $\mathring{\sigma}$ are defined by
\begin{align*}
\tilde{g}_{AB}dx^Adx^B&:=-\left(1-\frac{2M}{r} \right)dt^2+\left(1-\frac{2M}{r}\right)^{-1}dr^2\\
\mathring{\sigma}_{\alpha\beta}dx^\alpha dx^\beta&:=d\theta^2+\sin^2\theta d\phi^2.
\end{align*}
Their Levi-Civita connections are denoted by $\tilde{\nabla}$ and $\mathring{\nabla}$ separately and are extended trivially to tensor bundles on $\mathcal{M}$. In particular
\begin{equation*}
\begin{split}
&\tilde{\nabla}_A dx^B=-\tilde{\Gamma}^B_{AC}dx^C,\\
&\tilde{\nabla}_A dx^\alpha=0,\\
&\mathring{\nabla}_\alpha dx^\beta=-\mathring{\Gamma}^\beta_{\alpha\gamma}dx^\gamma,\\
&\mathring{\nabla}_\alpha dx^B=0.
\end{split}
\end{equation*}
The Schwarzschild spacetime has four linearly independent Killing vector fields,
\begin{align*}
T&=\frac{\partial}{\partial t},\ \Omega_1=\frac{\partial}{\partial \phi},\\
\Omega_2&=\cos\phi\frac{\partial}{\partial \theta}-\cot\theta\sin\phi\frac{\partial}{\partial \phi},\\
\Omega_3&=\sin\phi\frac{\partial}{\partial \theta}+\cot\theta\cos\phi\frac{\partial}{\partial \phi}.
\end{align*}
We use these Killing vectors as commutators and denote them by $\Gamma=\{T,\Omega_1,\Omega_2,\Omega_3\}$.\\

Synge \cite{Synge} and Kruskal \cite{Kruskal} showed the Schwarzschild metric can be extended analytically beyond $r=2M$ as a vacuum spacetime. The maximal analytic extension can be described by the Penrose diagram which represents a conformal compactification of the spacetime. \\
\begin{figure}[h]
\begin{center}
\centering
\includegraphics{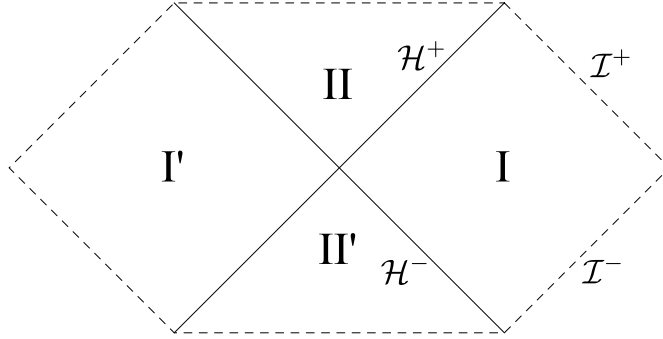}
\caption{Penrose diagram}
\end{center}
\end{figure}
The $(t,r,\theta,\phi)$ coordinate covers the exterior region I. In this paper we focus on region I together with the future horizon $\mathcal{H}^+$. On the horizon, it's convenient to work with the incoming Eddington-Finkelstein coordinate
\begin{align*}
v=t+r_*,\ R=r,
\end{align*}
where $r_*=r+2M\log (r-2M)-3M-2M\log M$. The spacetime metric is of the form
\begin{align*}
g=-\left(1-\frac{2M}{R}\right)dv^2+2dvdR+R^2\Big( d\theta^2+\sin^2\theta d\phi^2 \Big).
\end{align*}
The null vectors play an important role in Lorenzian geometry. In particular, they capture the radiation property of the wave equation. We fix the following two null vector fields:
\begin{align*}
L&=\frac{\partial }{\partial t}+\left(1-\frac{2M}{r}\right)\frac{\partial}{\partial r},\\
\underline{L}&=\frac{\partial }{\partial t}-\left(1-\frac{2M}{r}\right)\frac{\partial}{\partial r}.
\end{align*}
Both $L$ and $\underline{L}$ are smooth up to the future horizon. On the horizon, $L=2\frac{\partial}{\partial v}$ and $\underline{L}=0$.\\

We work with spacelike hypersurfaces defined as follows. Let $r_1$ be a fixed number in $(2M,3M)$. $\Sigma_0$ is a spherical symmetric spacelike hypersurface with the following properties: First, $\Sigma_0$ intersects the future horizon transversely. Second, $\Sigma_0\cap\{r\geq r_1\}=\{t=0,\ r\geq r_1\}$. We define the following notations: \\\\
$\Sigma_\tau$: the image of $\Sigma_0$ under the diffeomorphism $\Phi_\tau$ generated by $T$\\
$D(\tau_1,\tau_2)$: the domain bounded by $\Sigma_{\tau_1},\ \Sigma_{\tau_2}$ and $\mathcal{H}^+$\\
$n_{\Sigma_\tau}$: the unit future normal vector of $\Sigma_\tau$\\
$\frac{\partial}{\partial \rho}$: the radial tangent vector of $\Sigma_\tau$ such that $\frac{\partial r}{\partial \rho}=1$\\
$dVol_{\Sigma_\tau}$: the induced volume form of $\Sigma_\tau$\\
$dVol_{S^2}$: the volume form of the unit sphere.
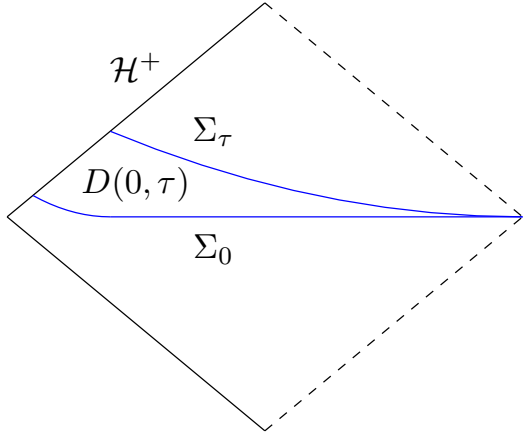
\begin{figure}[h]
\begin{center}
\begin{tikzpicture}[scale=1.2, transform shape]
\begin{axis}[
    hide axis,    
]

\addplot [
    domain=0:1, 
    samples=10,
]
{x};
 \addplot [
    domain=0:1, 
    samples=10,    
]
{-x};
\addplot [
    dashed,
    domain=1:2, 
    samples=10,         
]
{x-2};
\addplot [
    domain=1:2, 
    samples=10, 
    dashed,
]
{-x+2};
\addplot [
    domain=0.4:2, 
    samples=10, 
    blue,
]
{0};

\addplot [
    domain=0.1:0.4, 
    samples=10, 
    blue,
]
{10/9*(x-0.4)^2};

\addplot [
    domain=0.4:2, 
    samples=10, 
    blue,
]
{5/32*(x-2)^2};

\node at (axis cs:0.8,-0.15) {$\Sigma_0$};
\node at (axis cs:0.8,0.4) {$\Sigma_\tau$};
\node at (axis cs:0.5,0.15) {$D(0,\tau)$};
\node at (axis cs:0.5,0.7) {$\mathcal{H}^+$};
\end{axis}
\end{tikzpicture}
\caption{Foliation of spacelike hypersurfaces $\Sigma_\tau$}
\end{center}
\end{figure}
\\\\
For more information for the Schwarzschild spacetime, we refer readers to \cite{Wald3}.
\subsection{spherical vector bundles}
The main quantities of this paper would be sections of vector bundles $\mathcal{L}(-1)$ or $\mathcal{L}(-2)$ defined as follows.
\begin{definition}
$\mathcal{L}(-1)\subset T^*\mathcal{M}$ is the subbundle of spherical one forms. Locally, a section $\Phi$ of $\mathcal{L}(-1)$ can be written as
\begin{align*}
\Phi=\Phi_\alpha dx^\alpha.
\end{align*} 
$\mathcal{L}(-2)\subset T^*\mathcal{M}\otimes_s T^*\mathcal{M}$ is the subbundle of traceless spherical 2-tensors which are symmetric. Locally, a section $\Psi$ of $\mathcal{L}(-2)$ can be written as
\begin{align*}
\Psi=\Psi_{\alpha\beta} dx^\alpha dx^\beta,\ \mathring{\sigma}^{\alpha\beta}\Psi_{\alpha\beta}=0.
\end{align*}
The connections on $\mathcal{L}(-1)$ and $\mathcal{L}(-2)$ induced by the Levi-Civita connection are denoted by ${}^1\nabla$ and ${}^2\nabla$ separately.
\end{definition}
We omit the superscript in ${}^1 \nabla$ and ${}^2\nabla$ when it doesn't cause confusion. Alternatively, one can define $\mathcal{E}(-1)=T^* S^2$, and define $\mathcal{E}(-2)\subset T^* S^2\otimes_s T^* S^2$ to be the bundle of symmetric traceless tensors. Then $\mathcal{L}(-1)$ and $\mathcal{L}(-2)$ are the pull back bundle of $\mathcal{E}(-1)$ and $\mathcal{E}(-2)$ through the projection map. The eigensections of $\mathring{\Delta}$ on $\mathcal{E}(-1)$ and $\mathcal{E}(-2)$ are called spin-weighted spherical harmonics and can be expressed explicitly by scalar spherical harmonics. 
\begin{lemma}
Let $Y^{m\ell}(\theta,\phi)$ , $\ell\geq 0$ and $-\ell\leq m\leq \ell$, be the spherical harmonics on $S^2$. Define
\[Y^{m\ell}_\alpha:=\mathring{\nabla}_\alpha Y^{m\ell},\] 
and 
\[Z^{m\ell}_\alpha:=\mathring{\epsilon}_\alpha^{\ \beta}\mathring{\nabla}_\beta Y^{m\ell}.\] 
Then $\Big\{Y^{m\ell}_\alpha,Z^{m\ell}_\alpha\ \Big|\ell\geq 1 \Big\}$ forms a complete $L^2$ basis of eigensections of $\mathring{\Delta}$ on $\mathcal{E}(-1)$ with eigenvalues $-\ell (\ell+1)+1$.\\
Similarly, define
\[Y^{m\ell}_{\alpha\beta}:=\mathring{\nabla}_\alpha Y_\beta^{m\ell}+\mathring{\nabla}_\beta Y_\alpha^{m\ell}-\mathring{\nabla}^\gamma Y_\gamma^{m\ell}\mathring{\sigma}_{\alpha\beta},\] 
and 
\[Z^{m\ell}_{\alpha\beta}:=\mathring{\nabla}_\alpha Z_\beta^{m\ell}+\mathring{\nabla}_\beta Z_\alpha^{m\ell}.\] 
Then $\Big\{Y^{m\ell}_{\alpha\beta},Z^{m\ell}_{\alpha\beta}\ \Big|\ell\geq 2 \Big\}$  forms a complete $L^2$ basis of eigensections of $\mathring{\Delta}$ on $\mathcal{E}(-2)$ with eigenvalues $-\ell (\ell+1)+4$.
\end{lemma}

Any section $\Phi$ of $\mathcal{L}(-1)$ can be decomposed as
\begin{align*}
\Phi=\sum_{m,\ell} \Big(f_1^{m\ell}(t,r)Y^{m\ell}_{\alpha}+f_2^{m\ell}(t,r)Z^{m\ell}_{\alpha}\Big)dx^\alpha.
\end{align*} 
$\Phi$ is said to be supported on the angular mode $\ell=k$ if $f_{1}^{m\ell}=f_2^{m\ell}\equiv 0$ for any $\ell\neq k$ and the definition for $\Psi\in \Gamma(\mathcal{L}(-2))$ is the same.\\

We also define the angular operators $D: \mathcal{L}(-1)\longrightarrow \mathcal{L}(-2)$ and $D^\dagger:\mathcal{L}(-2)\longrightarrow \mathcal{L}(-1)$ as
\begin{align}
D(\Phi_\alpha dx^\alpha)&=r\Big(\mathring{\nabla}_\alpha \Phi_\beta+\mathring{\nabla}_\beta \Phi_\alpha-\mathring{\nabla}^\gamma \Phi_\gamma \mathring{\sigma}_{\alpha\beta}\Big)dx^\alpha dx^\beta,\\
D^\dagger (\Psi_{\alpha\beta} dx^\alpha dx^\beta )&=-\frac{2}{r}\mathring{\sigma}^{\gamma\beta}\mathring{\nabla}_{\gamma}\Psi_{\alpha\beta}dx^\alpha.
\end{align}
It's straightforward to verify the following properties of $D$ and $D^\dagger$.
\begin{lemma}
For any $\Phi\in\mathcal{L}(-1)$ and $\Psi\in\mathcal{L}(-2)$,
\begin{align*}
\int_{S^2(t,r)} D\Phi\cdot\Psi dVol_{S^2}&=\int_{S^2(t,r)} \Phi\cdot D^\dagger\Psi dVol_{S^2},\\
\int_{S^2(t,r)} |D\Phi|^2 dVol_{S^2}&=\int_{S^2(t,r)} 2r^2|\nablas \Phi|^2-2|\Phi|^2 dVol_{S^2},\\
\int_{S^2(t,r)} |D^\dagger\Psi|^2 dVol_{S^2}&=\int_{S^2(t,r)} 2r^2|\nablas \Psi|^2+4|\Psi|^2 dVol_{S^2}.
\end{align*}
Furthermore,
\begin{align*}
[D,\nabla_A]=0,\ [D^\dagger,\nabla_A]=0.
\end{align*}
\end{lemma}
For any section $\Phi$ of $\mathcal{L}(-1)$ or $\mathcal{L}(-2)$, its stress-energy tensor is defined by
\begin{align}
T_{ab}[\Phi]:=\nabla_a\Phi\cdot \nabla_b\Phi-\frac{1}{2}(\nabla_c\Phi \cdot\nabla^c\Phi )g_{ab}.
\end{align}
$T_{ab}[\Phi]$ satisfies the energy condition that $T_{ab}[\Phi]V^aW^b$ is non-negative as $V$ and $W$ are future timelike or null vectors. Since the curvatures of $\mathcal{L}(-1)$ and $\mathcal{L}(-2)$ are supported along the angular directions, for any vector field $W$ orthogonal to spheres,
\begin{align}
\nabla^a (T_{ab}[\Phi]W^a)=\Box\Phi\cdot \nabla_W\Phi+T_{ab}[\Phi]\nablab^a W^b.
\end{align} 
For any vector field $W$, we define 
\begin{align*}
e^W[\Phi]&:=T_{ab}[\Phi]W^b n_{\Sigma_\tau}^b.
\end{align*}
Following \cite{Ionescu-Klainerman}, we define the unweighted and weighted energies.
\begin{align}
e[\Phi]&:=|\nabla_L \Phi|^2+|\slashed{\nabla}\Phi|^2+|\nabla_{\rho}\Phi|^2+\frac{1}{r^2}|\Phi|^2,\\
e^p[\Phi]&:=r^p\left( |\nabla_L \Phi|^2+|\slashed{\nabla}\Phi|^2+\frac{1}{r^2}|\nabla_{\rho}\Phi|^2+\frac{1}{r^2}|\Phi|^2\right),\\
E^p[\Phi](\tau)&:=\int_{\Sigma_\tau}e^p[\Phi]dVol_{\Sigma_\tau},
\end{align}
where $|\nablas \Phi|^2$ consists of derivatives along the angular directions. 
\begin{align*}
|\nablas \Phi|^2:=\frac{1}{r^2}\left( |\nabla_\theta \Phi|^2+\frac{1}{\sin^2\theta}|\nabla_\phi \Phi|^2 \right).
\end{align*}
The degenerated weighted energy is defined by
\begin{align}
e^{p,deg}[\Phi]&:=r^p\left( \left(1-\frac{3M}{r}\right)^2\Big(|\nabla_L \Phi|^2+|\slashed{\nabla}\Phi|^2\Big)+\frac{1}{r^2}|\nabla_{\rho}\Phi|^2+\frac{1}{r^2}|\Phi|^2 \right),\\
E^{p,deg}[\Phi](\tau)&:=\int_{\Sigma_\tau}e^{p,deg}[\Phi]dVol_{\Sigma_\tau}.
\end{align}
We use the Killing vector fields $\Gamma=\{ T,\Omega_1,\ \Omega_2,\ \Omega_3\}$ as commutators to define higher order energies by
\begin{align}
E^{p}[\Phi]^{\leq s}(\tau)&:=\sum_{j=0}^s E^p[\mathcal{L}^j_\Gamma \Phi](\tau),\\
\bar{E}[\Phi]^{\leq s}&:=E^{p=2}[\Phi]^{\leq s}(0).
\end{align}
From the Sobolev embedding, we have the following $L^\infty$ bound:
\begin{lemma}\label{Linfty}
For any $p\geq 0$ and any section $\Phi$ of $\mathcal{L}(-1)$ or $\mathcal{L}(-2)$, we have
\begin{align}
|\Phi(\tau,r)|^2   \leq \frac{C}{r^{p}} E^{p}[\Phi]^{\leq 2}(\tau).
\end{align}
\end{lemma}
\begin{proof}
First we assume $\Phi$ is spatially compactly supported. From the Sobolev inequality on $S^2$,
\begin{align*}
|\Phi(\tau , r')|^2\leq C \int_{S^2(\tau ,r')} \sum_{j=0}^2 |\mathcal{L}^j_\Omega \Phi|^2 dVol_S^2.
\end{align*}
From the compactness of $\textup{supp}(\Phi)$ on $\Sigma_\tau$,
\begin{align*}
\int_{S^2(\tau,r')} |\Phi|^2 dVol_S^2&\leq \int_{r'}^\infty \int_{S^2(\tau,r)} \left(\frac{|\nabla_\rho \Phi|^2}{r^2}+\frac{|\Phi|^2}{r^2}\right)r^2dVol_{S^2} dr\\
&\leq C\int_{\Sigma_\tau\cap\{r\geq r'\}} \frac{1}{(r')^p} e^p[\Phi]dVol_{\Sigma_\tau}\\
&=\frac{C}{(r')^p} E^p[\Phi].
\end{align*}
Then the result follows by applying the same estimate to $\mathcal{L}_\Omega \Phi$ and $\mathcal{L}^2_\Omega\Phi$. The assumption of compact support can then be removed by approximation.
\end{proof}
\subsection{Lichnerowicz d'Alembertian for odd solutions}
We first give the definition of odd and even symmetric 2-tensors.
\begin{definition}
A symmetric 2-tensor $h=h_{ab}dx^adx^b$ is an odd tensor provided 
\begin{align*}
h_{AB}&=0,\\
\mathring{\nabla}^\alpha h_{A\alpha}&=0,\\
\mathring{\nabla}^\alpha\mathring{\nabla}^\beta h_{\alpha\beta}&=0,\ \mathring{\sigma}^{\alpha\beta}h_{\alpha\beta}=0.
\end{align*}
$h$ is an even tensor provided 
\begin{align*}
\mathring{\epsilon}^{\beta\alpha}\mathring{\nabla}_\beta h_{A\alpha}&=0,\\
\mathring{\epsilon}^{\beta\alpha} \mathring{\nabla}_{\beta}\mathring{\nabla}^\gamma h_{\alpha\gamma}&=0.
\end{align*}
\end{definition}
From the Hodge decomposition, any symmetric $2$-tensor $h$ can be decomposed into even and odd parts $h=h_1+h_2$ and $h$ is a solution of (\ref{LVEE}) if and only if $h_1$ and $h_2$ are both solutions. Thus one can study them separately. From now on we assume $h$ is an odd tensor through the entire paper. For any odd solution of (\ref{LVEE}), one can define two gauge-invariant quantities $P$ and $Q$ as
\begin{align}
P&:=r^3\epsilon^{AB}\tilde{\nabla}_B (r^{-2}h_{A\alpha})dx^\alpha,\\
Q&:=\Big(\tilde{\nabla}^A r\Big)\left(\mathring{\nabla}_\alpha h_{A\beta}+\mathring{\nabla}_\beta h_{A\alpha}-r^2\tilde{\nabla}_A (r^{-2}h_{\alpha\beta})\right) dx^\alpha dx^\beta.
\end{align}
They satisfy the Regge-Wheeler equation \cite{Regge-Wheeler,Hung-Keller-Wang}
\begin{equation}\label{Regge Wheeler}
\begin{split}
{}^1\Box P-V_PP&=0,\ V_P=\frac{1}{r^2}\left(1-\frac{8M}{r}\right),\\
{}^2\Box Q-V_QQ&=0,\ V_Q=\frac{4}{r^2}\left(1-\frac{2M}{r}\right).\\
\end{split}
\end{equation}
For an odd tensor $h$, we define the following three sections:
\begin{align}
H_0:=h_{t\alpha}dx^\alpha,\ H_1:=\left(1-\frac{2M}{r}\right)h_{r\alpha}dx^\alpha\in\ \Gamma(\mathcal{L}(-1)),
\end{align} 
and
\begin{align}
H_2=-h_{\alpha\beta}dx^\alpha dx^\beta \in \ \Gamma(\mathcal{L}(-2)).
\end{align}
Under this notation, the harmonic gauge condition (\ref{HG}) can be written as
\begin{align}\label{wavegauge}
-\left(1-\frac{2M}{r}\right)^{-1}\nabla_t H_0+\nabla_r H_1+\frac{1}{2r}D^\dagger H_2+\frac{3}{r}H_1=0.
\end{align}
Also, $P$ and $Q$ can be expressed as
\begin{align}\label{def_P}
P=\left(1-\frac{2M}{r}\right)^{-1}r\nabla_t H_1-r\nabla_r H_0+H_0,
\end{align}
and
\begin{align}\label{def_Q}
Q=\frac{1}{r}DH_1+\left(1-\frac{2M}{r}\right)\nabla_r H_2.
\end{align}
\begin{lemma}\label{mainequation}
Under the above notation, (\ref{LVEEG}) can be written as
\begin{align}\label{equation_H0}
{}^1\Box H_0=V_0H_0-\frac{2M}{r^3}P,
\end{align}
and
\begin{equation}\label{eqiation_H1H2}
\begin{split}
{}^1\Box H_1&=V_1H_1+\frac{1}{r^2}\left(1-\frac{3M}{r}\right)D^\dagger H_2,\\
{}^2\Box H_2&=V_2H_2+\frac{2}{r^2}DH_1,
\end{split}
\end{equation}
where $V_0=\frac{1}{r^2}\left(1-\frac{2M}{r}\right),\ V_1=\frac{1}{r^2}\left( 5-\frac{18M}{r} \right)$ and $V_2=\frac{2}{r^2}$.
\end{lemma}
The derivation of these equations is put in the appendix. We adapt the notation that $A\lesssim B$ if $A\leq CB$ for some constant $C$ and that $A\leq_s B$ if
\begin{align*}
\int_{S(t,r)} A dVol_{S^2}\leq \int_{S(t,r)} B dVol_{S^2}.
\end{align*}
For example, the Poincar\'e inequality implies $\frac{1}{r^2}|\Phi|^2\lesssim_s |\nablas \Phi|$. We end this section by the precise statement of the main theorem.
\begin{theorem}[precise version]\label{thm_main}
Let $h$ be an odd solution of (\ref{LVEEG}) under the harmonic gauge condition (\ref{HG}). Define
\begin{align*}
\hat{h}:=h-\sum_{m=-1,0,1}d_m K_m,
\end{align*}
where $d_m$ are three numbers determined by $h$ and $K_m$ correspond to linearlized angular momentum solutions(see section \ref{ddd}). Let $H_0,H_1,H_2,P$ and $Q$ be the sections derived from $\hat{h}$. For any $0<\delta<\frac{1}{8}$, denote $m=m(\delta)=\left \lceil -\log_2\delta \right \rceil+1$. Then there exists a constant $C$ which depends only on $\delta$ such that for any $p\in [\delta,2-\delta]$, 
\begin{align*}
E^p[H_0](\tau)&\leq CI_0\tau^{-2+p+\delta},\\
E^p[H_1,H_2](\tau)&\leq CI_1\tau^{-2+p+\delta},
\end{align*}
where
\begin{align*}\\
I_0&=\bar{E}[H_0]^{\leq 2+m}+\bar{E}[P]^{\leq 5+m},\\
I_1&=\bar{E}[H_1,H_2]^{\leq 2}+\bar{E}[Q]^{\leq 5}+\bar{E}[H_0]^{\leq 5+m}+\bar{E}[P]^{\leq 8+m}.
\end{align*}
\end{theorem}
Together with Lemma \ref{Linfty}, we immediately have the following corollary:
\begin{corollary}
Under the same assumption as above, we have
\begin{align*}
\|r^{p/2} H_0 \|^2_{L^\infty (\Sigma_\tau)}&\leq C I_2\tau^{-2+p+\delta},\\
\|r^{p/2} H_1 \|^2_{L^\infty (\Sigma_\tau)}+\|r^{p/2} H_2\|^2_{L^\infty (\Sigma_\tau)}&\leq CI_3\tau^{-2+p+\delta},
\end{align*}
where
\begin{align*}\\
I_2&=\bar{E}[H_0]^{\leq 4+m}+\bar{E}[P]^{\leq 7+m},\\
I_3&=\bar{E}[H_1,H_2]^{\leq 4}+\bar{E}[Q]^{\leq 7}+\bar{E}[H_0]^{\leq 7+m}+\bar{E}[P]^{\leq 10+m}.
\end{align*}
\end{corollary}

\section{Estimate for $H_0$}
In this section we prove the decay estimates for $H_0$ based on (\ref{equation_H0}) and $P$ is treated as an error term. We use the vector field method and construct currents $J_i[H_0]_a$ with certain positivity. In subsection \ref{sr_H0} and \ref{sM_H0}, we use the red-shift vector and the Morawetz vector introduced by Dafermos and Rodnianski \cite{Dafermos-Rodnianski}. The $r^p$-hierarchy estimate of Dafermos and Rodnianski \cite{Dafermos-Rodnianski2} is used in subsection \ref{sp_H0}. In subsection \ref{sc_H0}, we combine the previous currents in the manner of \cite{Ionescu-Klainerman}. The decay of $H_0$ is presented in subsection \ref{sd_H0}. The construction of this section holds for any spherical harmonic number $\ell\geq 1$.
\subsection{$T$ vector field}\label{sT_H0}
By using $T=\frac{\partial}{\partial t}$ as a multiplier, we define
\begin{align}
J_1[H_0]_a&:=\Big(T_{ab}[H_0]-\frac{1}{2}V_0 (\nabla H_0)^2 g_{ab}\Big)T^b,\\
e_1[H_0]     &:=J_1[H_0]\cdot n_{\Sigma_\tau}.
\end{align}
From direct computation,
\begin{align*}
\Div J_1[H_0]&=\nabla_T H_0\cdot \left( -\frac{2M}{r^3}P\right).
\end{align*}
From the Poincar\'e inequality,
\begin{align*}
e_1[H_0]&\approx_s e^T[H_0] \approx  |\nabla_L H_0|^2+\left(1-\frac{2M}{r}\right)|\nabla_\rho H_0|^2+|\nablas H_0|^2.
\end{align*}
Because $T$ is a null vector on the horizon, there is degeneracy in the $\frac{\partial}{\partial \rho}$ direction. This degeneracy can be removed by the red-shift vector.
\subsection{red-shift vector field}\label{sr_H0}
Consider a vector $Y$ such that
\begin{align*}
Y\Big|_{r=2M}&=-2\frac{\partial}{\partial R},\\
\nabla_Y Y\Big|_{r=2M}&=-sT-\sigma Y,
\end{align*}
for some $s,\sigma>0$ to be determined. The red-shift current $J_2[H_0]$ for $H_0$ is defined by
\begin{align}
J_{2}[H_0]_a&:=T_{ab}[H_0] Y^b,\\
K_2[H_0]&:=\Div  J_2[H_0],\\
e_2[H_0]&:=J_2[H_0]\cdot n_{\Sigma_\tau}.
\end{align}
We denote the contribution of $P$ in $K_2[H_0]$ by $Err_2[H_0]$.
\begin{align*}
K_2[H_0]&=\mathring{K_2}[H_0]+Err_2[H_0],\\
Err_2[H_0]&=\nabla_Y H_0\cdot \left(-\frac{2M}{r^3}P\right).
\end{align*}
From (\ref{g_redshift}), on the horizon $r=2M$ we have
\begin{align*}
\mathring{K}_2[H_0]=&\left(\frac{s}{2}|\nabla_v H_0|^2+\frac{1}{2M}|\nabla_R H_0|^2+\frac{2}{M}\nabla_R H_0\cdot\nabla_v H_0+\frac{\sigma}{2}|\nablas H_0|^2\right)\\
                 &+V_0H_0\cdot\nabla_Y H_0.
\end{align*}
From the Poincar\'e inequality, one can pick $s=\sigma$ large enough such that $\mathring{K}_2[H_0]\gtrsim_s e[H_0]$ on the horizon. By extending $Y$ smoothly to the exterior region such that $Y$ is non-spacelike and $Y=0$ as $r\geq r_1$, we obtain
\begin{align}
\mathring{K}_2[H_0]\gtrsim_s e[H_0]\ \textup{as}\ 2M\leq r\leq r_0
\end{align}
for some $r_0\in(2M,r_1)$ from a continuity argument. In summary, we get
\begin{lemma}
\begin{align*}
\begin{array}{cc}
\mathring{K}_2[H_0] \gtrsim_s e[H_0]\ & \textup{in}\ [2M,r_0],\\
\Big|\mathring{K}_2[H_0]\Big| \lesssim e[H_0]\ & \textup{in}\ [r_0,r_1],\\
\mathring{K}_2[H_0] = 0\ & \textup{in}\  [r_1,\infty ).
\end{array}
\end{align*}
For the boundary term, $e_2[H_0]$ is non-negative and
\begin{equation*}
\begin{split}
e_2[H_0]+e_1[H_0]&\gtrsim_s e[H_0].
\end{split}
\end{equation*}
\end{lemma}
\subsection{Morawetz vector field}\label{sM_H0}
For any radial function $f(r)$, let $X=f(r)\frac{\partial}{\partial r_*}=f(r)\left(1-\frac{2M}{r}\right)\frac{\partial}{\partial r}$. Define $g(r):=\frac{1}{r^3}\left(1-\frac{2M}{r}\right) \left(1-\frac{3M}{r} \right)^2$ and $\epsilon_1>0$ be a small positive number to be determined. Let $\omega^X=\left(1-\frac{2M}{r}\right)\left(\frac{\partial f}{\partial r}+\frac{2f}{r}\right)$ and define the Morawetz current as
\begin{equation}
\begin{split}
J_3[H_0]_a:=&\Big(T_{ab}[H_0]-\frac{1}{2}V_0|H_0|^2\Big)X_b+\frac{1}{4}\omega^X \nablab_a |H_0|^2-\frac{1}{4}\nablab_a \omega^X |H_0|^2\\
           &-\frac{1}{2}\epsilon_1 g\nabla_a |H_0|^2+\frac{1}{2}\nabla_a (\epsilon_1 g)|H_0|^2,\\
\end{split}
\end{equation}
\begin{align}           
K_3[H_0]&:=\Div J_3[H_0],\\
e_3[H_0]&:=J_3[H_0]\cdot n_{\Sigma_\tau}.
\end{align}
From (\ref{g_Morawetz}), $K_3[H_0]$ is of the form
\begin{align*}
{K}_3[H_0]=&\mathring{K}_3[H_0]+Err_3[H_0],\\
\mathring{K}_3[H_0]=&\left(1-\frac{2M}{r}\right)\left(\left(1-\frac{2M}{r}\right)\frac{\partial f}{\partial r}-\epsilon_1 g\right)|\nabla_{r}H_0|^2+\epsilon_1 g\left(1-\frac{2M}{r}\right)^{-1}|\nabla_{ t}H_0|^2\\
&+\left(\frac{f}{r}\left(1-\frac{3M}{r}\right)-\epsilon_1 g\right)|\slashed{\nabla}H_0|^2+\left(-\frac{1}{4} \Box\omega^X -\frac{1}{2}f\left(1-\frac{2M}{r}\right)\frac{\partial V_0}{\partial r}-\frac{M}{r^2}fV_0 \right)|H_0|^2\\
&+\epsilon_1\left(\frac{1}{2}\Box g-gV_0\right)|H_0|^2,\\
Err_3[H_0]=&\left(\nabla_X H_0+\frac{1}{2}\omega^X H_0-\epsilon_1g H_0\right)\cdot \left(-\frac{2M}{r^3} P\right).
\end{align*}
We use the multiplier $f=(1+\frac{3M}{2r})^2\left(1-\frac{3M}{r}\right)$ and then $\mathring{K}_3[H_0]\geq_s 0$ as $\epsilon_1=0$. In particular, after using the Poincar\'e inequality,
\begin{align*}
\mathring{K}_3[H_0]&\geq_s \left(-\frac{1}{4} \Box\omega^X -\frac{1}{2}f\left(1-\frac{2M}{r}\right)\frac{\partial V_0}{\partial r}-\frac{M}{r^2}f V_0 +\frac{f}{r^3}\left(1-\frac{3M}{r}\right)\right)|H_0|^2\\
                 &=\frac{1}{4r^3}\left( 8-24\left(\frac{M}{r}\right)-54\left(\frac{M}{r}\right)^2+108\left(\frac{M}{r}\right)^3+513\left(\frac{M}{r}\right)^4-891\left(\frac{M}{r}\right)^5 \right)|H_0|^2,                 
\end{align*}
which is positive in $[2M,\infty)$. Since $f -1\approx -\frac{1}{r^2},\ \frac{\partial f}{\partial r}\approx \frac{1}{r^3},\ \omega^X\approx \frac{2}{r}$ near the spatial infinity, by taking  $\epsilon_1>0$ small enough, one can make $\mathring{K}_3[H_0]$ coercive for bounded $r$ except the degeneracy at the photon sphere and at the horizon. In summary, we have
\begin{lemma}
$\mathring{K}_3[H_0]$ is non-negative. Moreover, in $[r_0,\infty)$,
\begin{align*}
\mathring{K}_3[H_0]\gtrsim_s& \frac{1}{r^3}|\nabla_r H_0|^2+\frac{1}{r^3}\left(1-\frac{3M}{r}\right)^2|\nabla_t H_0|^2+\frac{1}{r}\left(1-\frac{3M}{r}\right)^2|\nablas H_0|^2+\frac{1}{r^3}|H_0|^2.
\end{align*}
\end{lemma}
The next lemma concerns the boundary terms.
\begin{lemma}\label{H0_Morawetz_boundary}
There exists $C_0\geq 1$ such that we have the following estimates:
\begin{align*}
e^X[H_0]\geq_s -C_0 e_1[H_0].
\end{align*}
For any $R_1$ large enough, 
\begin{align*}
\left|\int_{\Sigma_\tau} e_3[H_0]-e^X[H_0] dVol_{\Sigma_\tau}\right|\leq &C_0\int_{\Sigma_\tau\cap\{2M\leq r\leq R_1\}} e_1[H_0]dVol_{\Sigma_\tau}\\
                                                                    +&C_0\int_{\Sigma_\tau\cap\{R_1\leq r\}} e^{L}[H_0]dVol_{\Sigma_\tau}.
\end{align*}
\end{lemma}
\begin{proof}
Because $\left(1-\frac{2M}{r}\right)\frac{\partial}{\partial r}+\frac{\partial}{\partial t}$ is a null vector, $e^{X+|f|T}[H_0]\geq 0$. Then the first inequality follows since $e_1$ is comparable to $e^T$. For the second inequality, note that
\begin{align*}
e_3[H_0]-e^X[H_0]=-\frac{1}{2}V_0 |H_0|^2 X\cdot n_{\Sigma_\tau}+\frac{1}{2}(\omega^X-2\epsilon_1 g) H_0\cdot\nabla_{n_{\Sigma_\tau}}H_0-\frac{1}{4}\nablab_{n_{\Sigma_\tau}}(\omega^X-2\epsilon_1 g)  |H_0|^2.
\end{align*}
For bounded $r$, each term is bounded by $e^T[H_0]$. For $r$ large, the only nonzero term is $\frac{1}{2}(\omega^X-2\epsilon_1 g) H_0\cdot\nabla_{n_{\Sigma_\tau}}H_0$. 
\begin{align*}
&\frac{1}{2}(\omega^X-2\epsilon_1 g) H_0\cdot\nabla_{n_{\Sigma_\tau}}H_0\\
=& \frac{1}{2}\D^{-1/2}(\omega^X-2\epsilon_1 g) H_0\cdot\nabla_{t}H_0\\
                                                   =&\frac{1}{2}\D^{-1/2}(\omega^X-2\epsilon_1 g) H_0\cdot\nabla_{L}H_0-\frac{1}{2}\left(1-\frac{2M}{r}\right)^{1/2}(\omega^X-2\epsilon_1 g) H_0\cdot\nabla_{r}H_0\\
                                                   =& \frac{1}{r}H_0\cdot\nabla_L H_0-\frac{1}{2r}\D^{1/2}\nablab_r|H_0|^2+l.o.t.
\end{align*}
The first term is controlled by $e^L[H_0]$ as
\begin{align*}
e^L[H_0]&=\frac{1}{2}\left(\D^{-1/2}|\nabla_L H_0|^2+\D^{1/2}|\nablas H_0|^2\right)\\
      &\geq_s\frac{1}{2}\left(\D^{-1/2}|\nabla_L H_0|^2+\frac{1}{r^2}\D^{1/2}|H_0|^2\right)\\
      &\gtrsim \frac{1}{r}|\nabla_L H_0||H_0|.
\end{align*}
We integrate by parts the last term,
\begin{align*}
&\left| \int_{\Sigma_\tau\cap\{R_1\leq r\}} \frac{1}{2r}\D^{1/2}\nablab_r|H_0|^2 dVol_{\Sigma_\tau}\right|\\
= &\left| \int_{\Sigma_\tau\cap\{R_1\leq r\}} \frac{1}{2r^2}\D^{1/2}|H_0|^2 dVol_{\Sigma_\tau}+ \textup{boundary term}\right|\\
\lesssim_s &\int_{\Sigma_\tau\cap\{R_1\leq r\}} e^L[H_0] dVol_{\Sigma_\tau}+ \textup{boundary term}.
\end{align*}
By absorbing the boundary term into $e^T[H_0]$, the result follows.
\end{proof}
\subsection{$r^p$-hierarchy}\label{sp_H0}
Fix two small numbers $\delta_1,\delta_2>0$. The constants below can depend on $\delta_1$ and $\delta_2$. For any $p\in [\delta_1,2-\delta_2]$, define
\begin{align}
\tilde{J}^p_4[H_0]_a:=\frac{r^p}{r^2\left(1-\frac{2M}{r}\right)}T_{ab}[rH_0]L^b-\frac{Mr^p}{r^3\left(1-\frac{2M}{r}\right)}|rH_0|^2L_a-\frac{1}{2}\frac{r^p}{r^2\left(1-\frac{2M}{r}\right)}V_0 |rH_0|^2L_a.
\end{align}
From (\ref{g_rp}),
\begin{align*}
&\Div \tilde{J}_4^p[H_0]\\
=&\left(1-\frac{2M}{r}\right)^{-2}r^{p-3}\left( \frac{p}{2}\left(1-\frac{2M}{r}\right)-\frac{M}{r} \right)|\nabla_L (rH_0)|^2+\left(1-\frac{p}{2}\right)r^{p-3}|\slashed{\nabla}(rH_0)|^2\\
                     &+\left( (3-p)Mr^{p-6}-\frac{r^p}{r^2\left(1-\frac{2M}{r}\right)}\frac{1}{r}\left(1-\frac{M}{r}\right)V_0-\left(1-\frac{2M}{r}\right)\partial_r\left( \frac{r^p}{2r^2 \D}V_0 \right) \right)|(rH_0)|^2\\
                     &+\frac{r^p}{r^2\D}\Big(\nabla_L (rH_0)\Big)\cdot r\left(-\frac{2M}{r^3}P\right).
\end{align*}
Because $V_0\approx 1/r^2$ and $\delta_1\leq p\leq 2-\delta_2$,
\begin{align*}
\Div \tilde{J}_4^p[H_0] &\geq \frac{\delta_1}{2}r^{p-3}|\nabla_L (rH_0)|^2+\frac{\delta_2}{2}r^{p-3}|\nablas (rH_0)|^2+\frac{\delta_2}{2}r^{p-5}|rH_0|^2\\
                        &+\frac{r^p}{r^2\D}\Big(\nabla_L (rH_0)\Big)\cdot r\left(-\frac{2M}{r^3}P\right)+l.o.t.\\
                        &\gtrsim r^{p-1}|\nabla_L H_0|^2+r^{p-1}|\nablas H_0|^2+r^{p-3}|H_0|^2+\frac{r^p}{r^2(1-\frac{2M}{r})}\Big(\nabla_L (rH_0)\Big)\cdot r\left(-\frac{2M}{r^3}P\right),
\end{align*} 
as $r$ is large enough. Compared to $e^{p-1}[H_0]$, the derivative along the $\frac{\partial}{\partial r}$ direction is missing. To obtain the full $e^{p-1}[H_0]$, we consider
\begin{align*}
\nablab^a \bigg( r^{p-2}\Big(T_{ab}[H_0]-\frac{1}{2}V_0|H_0|^2g_{ab}\Big)T^b\bigg)&=(p-2)r^{p-3}T_{ab}[H_0]T^b\nablab^a r+r^{p-2}\nabla_T H_0\cdot\left(-\frac{2M}{r^3}P\right)\\
        &=\frac{2-p}{4}r^{p-3}\Big( |\nabla_{\underline{L}}H_0|^2-|\nabla_L H_0|^2 \Big)+r^{p-2}\nabla_T H_0\cdot\left(-\frac{2M}{r^3}P\right).
\end{align*}
The negative part $r^{p-3}|\nabla_L H_0|^2$ is much smaller than $r^{p-1}|\nabla_L H_0|^2$ in $\Div \tilde{J}^p_4[H_0]$. Hence we take $R_0\geq 5M$ large and a cut-off function $\eta(r)$ such that $\eta(r)=1$ for $r\geq R_0$ and $\eta=0$ for $r\leq R_0-M$. The $r^p$ current for $H_0$ is defined by
\begin{align}
J^p_4[H_0]&:=\eta \bigg(\tilde{J}^p_4[H_0]+r^{p-2}\Big(J^T[H_0]-\frac{1}{2}V_0|H_0|^2T\Big)\bigg),\\
K^p_4[H_0]&:=\Div J^p_4[H_0],\\
e^p_4[H_0]&:=J^p_4[H_0]\cdot n_{\Sigma_\tau}.          
\end{align}
As before, we single out the contribution of $P$ in $K^p_4[H_0]$ as $Err^p_4[H_0]$.
\begin{align*}
K^p_4[H_0]&=\mathring{K}^p_4[H_0]+Err^p_4[H_0],\\
Err^p_4[H_0]&=\eta\left(\frac{r^p}{r(1-\frac{2M}{r})}\nabla_L \Big(rH_0\Big)+r^{p-2}\nabla_T H_0   \right)\cdot\left(-\frac{2M}{r^3}P\right).
\end{align*}
From the above discussion, as $R_0$ is large enough, we have
\begin{lemma} 
There exists a constant $C_1>0$ such that
\begin{align*}
\begin{array}{cc}
\mathring{K}^p_4[H_0]\geq \frac{1}{C_1} e^{p-1}[H_0] & \textup{in}\ (R_0,\infty),\\
\Big|\mathring{K}^p_4[H_0]\Big|\lesssim e[H_0] & \textup{in}\ [R_0-M,R_0],\\
\mathring{K}^p_4[H_0]=0 & \textup{in}\ [2M,R_0-M).
\end{array}
\end{align*} 
For the boundary term, $e^p_4[H_0]$ is non-negative and
\begin{align*}
e^p_4[H_0]&\approx e^p[H_0]\ \textup{in}\ [R_0,\infty).
\end{align*}
\end{lemma}
\subsection{combination}\label{sc_H0}
We take $\epsilon_2>0$ small such that for all $p\in [\delta_1,2-\delta_2]$,
\begin{align*}
\begin{array}{cc}
\mathring{K}_3[H_0]\geq 2\epsilon_2 \Big|\mathring{K}^p_4[H_0]\Big| & \textup{in}\ [R_0-M,R_0].
\end{array}
\end{align*}
Then we pick $R_1\geq R_0$ such that 
\begin{align*}
\frac{3C_0}{\log 2}\leq \frac{1}{2C_1}\epsilon_2 R_1^{\delta_1},
\end{align*}
and
\begin{align*}
\begin{array}{cc}
\epsilon_2e^{p=\delta_1}_4[H_0]\geq 2C_0 e^L[H_0] & \textup{in}\ [R_1,\infty).
\end{array}
\end{align*}
We consider a non-increasing function $h(r)$ satisfying
\begin{align*}
\begin{array}{cc}
h(r)= 3C_0 &\textup{in}\ [2M,R_1],\\
h(r)= 1 &\textup{in}\ [2R_1,\infty),\\
|h'(r)|\leq \frac{3C_0}{\log 2}r^{-1} &\textup{in}\ [R_1,2R_1].
\end{array}
\end{align*}
The fifth current is defined by
\begin{align}
J_5[H_0]_a&:= \left(T_{ab}[H_0]-\frac{1}{2}V_0|H_0|^2g_{ab}\right)\Big( h(r)T^b \Big),\\
K_5[H_0]&:=\Div J_5[H_0],\\
e_5[H_0]&:=J_5[H_0]\cdot n_{\Sigma_\tau}.
\end{align}
\begin{align*}
K_5[H_0]&=\mathring{K}_5[H_0]+Err_5[H_0],\\
Err_5[H_0]&=h\nabla_T H_0\cdot\left(-\frac{2M}{r^3}P\right).
\end{align*}
Then we have
\begin{lemma}
\begin{align*}
\begin{array}{cc}
\mathring{K}_5[H_0]=0 & \textup{in}\ [2M,R_1]\cup [2R_1,\infty),\\
\mathring{K}_5[H_0]\geq -\frac{3C_0}{(\log 2) r}|\nabla_L H_0|^2 & \textup{in}\ [R_1,2R_1].
\end{array}
\end{align*}
\begin{align*}
\begin{array}{cc}
e_5[H_0]= 3C_0 e_1[H_0] & \textup{in}\ [2M,R_1],\\
e_5[H_0]\geq e^T[H_0] & \textup{in}\ [R_1,\infty).\\
\end{array}
\end{align*}
\end{lemma}
We pick $\epsilon_3>0$ such that
\begin{align*}
\begin{array}{cc}
\mathring{K}_3[H_0]\geq \epsilon_3\Big|\mathring{K}_2[H_0]\Big| & \textup{in}\ [r_0,r_1].
\end{array}
\end{align*}
Then the final current for $H_0$ is defined by
\begin{align}
J^p[H_0]&:=\epsilon_3 J_2[H_0]+J_3[H_0]+\epsilon_2 J^p_4[H_0]+J_5[H_0],\\
K^{p}[H_0]&:=\Div J^p[H_0].
\end{align}
Again, we denote the contribution of $P$ in $K^p[H_0]$ by $Err^p[H_0]$.
\begin{align*}
K^{p}[H_0]=&\mathring{K}^{p}[H_0]+Err^p[H_0],\\
\mathring{K}^p[H_0]=&\epsilon_3 \mathring{K}_2[H_0]+\mathring{K}_3[H_0]+\epsilon_2 \mathring{K}^p_4[H_0]+\mathring{K}_5[H_0],\\
Err^p[H_0]=&\left( \epsilon_3\nabla_Y H_0+\nabla_X H_0+\frac{1}{2}(\omega^X-2\epsilon_1 g) H_0\right.\\
          &+\epsilon_2\eta\Big(r^{p-1}\nabla_L(rH_0)+r^{p-2}\nabla_T H_0 \Big)+h\nabla_T H_0 \bigg)\cdot\left(-\frac{2M}{r^3}P\right).\\
\end{align*}

\begin{lemma}\label{H_0_final_current}
By the above construction, we have
\begin{align*}
\mathring{K}^p[H_0]\gtrsim_s e^{p-1,deg}[H_0],
\end{align*}
and
\begin{align*}
\int_{\Sigma_\tau} J^p[H_0]\cdot n_{\Sigma_\tau} dVol_{\Sigma_\tau}\approx \int_{\Sigma_\tau} e^{p}[H_0] dVol_{\Sigma_\tau}.
\end{align*}
\end{lemma}
\begin{proof} We discuss the positivity of $\mathring{K}^p[H_0]$ in different intervals separately.\\
In $[2R_1,\infty)$, each term in $\mathring{K}^p[H_0]$ is non-negative and $\mathring{K}^p_4[H_0]\gtrsim e^{p-1}[H_0]$.\\
In $[R_1, 2R_1)$, the only negative term is $\mathring{K}_5$ and from the choice of $h(r)$,
\begin{align*}
\mathring{K}_5[H_0]+\epsilon_2 \mathring{K}_4[H_0]&\geq \left(-\frac{3C_0}{\log 2}r^{-p}+\frac{\epsilon_2}{C_3} \right)r^{p-1}|\nabla_L H_0|^2\\
                                                  &\geq \frac{\epsilon_2}{2C_1}r^{p-1}|\nabla_L H_0|^2.
\end{align*}
In $[R_0, R_1)$, each term is non-negative and $\mathring{K}^p_4[H_0]\gtrsim e^{p-1}[H_0]$.\\
In $[R_0-M, R_0)$, the only negative term is $\epsilon_2\mathring{K}_4[H_0]$ and $\mathring{K}_3[H_0]+\epsilon_2\mathring{K}_4[H_0]\geq \frac{1}{2}\mathring{K}_3[H_0]$ which is comparable to $e^{p-1}[H_0]$.\\
In $[r_1, R_0-M)$, each term is non-negative and $\mathring{K}_3[H_0]\gtrsim e^{p-1,deg}[H_0]$.\\
In $[r_0, r_1)$, the only negative term is $\epsilon_3\mathring{K}_2[H_0]$ and $\epsilon_3\mathring{K}_2[H_0]+\mathring{K}_3[H_0]\geq \frac{1}{2}\mathring{K}_3[H_0]$ which is comparable to $e^{p-1}[H_0]$.\\
In $[2M, r_0)$, each term is non-negative and $\mathring{K}_2[H_0]\gtrsim e^{p-1}[H_0]$.\\

To see the boundary term is comparable with $E^p[H_0]$, we note that
\begin{align*}
&\int_{\Sigma_\tau} e_3[H_0]+e_5[H_0]+\epsilon_2 e^{p}_4[H_0] dVol_{\Sigma_\tau}\\
\geq& -\left|\int_{\Sigma_\tau} e_3[H_0]-e^X[H_0] dVol_{\Sigma_\tau}\right|+\int_{\Sigma_\tau} e^X[H_0]+e_5[H_0]+\epsilon_2 e^{p}_4[H_0] dVol_{\Sigma_\tau}\\
       \geq& -C_0\int_{\Sigma_{\tau}\cap\{r\leq R_1\}} e_1[H_0] dVol_{\Sigma_\tau}-C_0\int_{\Sigma_\tau\cap\{r\geq R_1\}} e^{L}[H_0]dVol_{\Sigma_\tau}\\
       &+    2C_0\int_{\Sigma_{\tau}\cap\{r\leq R_1\}} e_1[H_0] dVol_{\Sigma_\tau}+\int_{\Sigma_\tau\cap\{r\geq R_1\}} e^{L}[H_0]+\epsilon_2 e^{p}_4[H_0]dVol_{\Sigma_\tau}\\
       \geq& C_0\int_{\Sigma_{\tau}\cap\{r\leq R_1\}} e_1[H_0] dVol_{\Sigma_\tau}+\int_{\Sigma_\tau\cap\{r\geq R_1\}} e^{L}[H_0]+\frac{1}{2}\epsilon_2 e^{p}_4[H_0] dVol_{\Sigma_\tau}.
\end{align*}
Therefore
\begin{align*}
\int_{\Sigma_\tau} J^p[H_0]\cdot n_{\Sigma_\tau} dVol_{\Sigma_\tau}\gtrsim E^p[H_0].
\end{align*}
To get the upper bound, we note that for $r\geq 2R_1$,
\begin{align*}
&e^X[H_0]+e_5[H_0] \lesssim_s  e^{X+T}[H_0]\\
=&fe^L[H_0]+(1-f)e^T[H_0] \lesssim e^{p=0}[H_0].
\end{align*}
\end{proof}
By integrating $K^p[H_0]$ in $D(\tau_1,\tau_2)$, we obtain
\begin{proposition}
For any $\tau_2>\tau_1\geq 0$ and $p\in [\delta_1,2-\delta_2],$ we have
\begin{equation}\label{H_0 hierarchy}
\begin{split}
&E^p[H_0](\tau_2)+\int_{\tau_1}^{\tau_2}E^{p-1,deg}[H_0](\tau)d\tau\\
\leq &C\left( E^p[H_0](\tau_1)+\int_{\tau_1}^{\tau_2}\int_{\Sigma_\tau} r^{p/2}\sqrt{e^p[H_0]}\cdot\frac{|P|}{r^3}  dVol_{\Sigma_\tau} d\tau \right).
\end{split}
\end{equation}
\end{proposition}
A similar construction can be done for $P$ and $Q$ from equation (\ref{Regge Wheeler}) without the error terms.
\begin{proposition}
For any $\tau_2>\tau_1\geq 0$ and $p\in [\delta_1,2-\delta_2],$ we have
\begin{align}\label{P hierarchy}
E^p[P](\tau_2)+\int_{\tau_1}^{\tau_2}E^{p-1,deg}[P]d\tau\leq CE^p[P](\tau_1),
\end{align}
\begin{align}
E^p[Q](\tau_2)+\int_{\tau_1}^{\tau_2}E^{p-1,deg}[Q]d\tau\leq CE^p[Q](\tau_1).
\end{align}
\end{proposition}
\subsection{decay estimate}\label{sd_H0}
From now on we fix a small number $\frac{1}{8}>\delta>0$. A direct corollary of (\ref{P hierarchy}) is the decay estimate for $E^p[P]$ and $E^p[Q]$.
\begin{corollary}\label{thm_PQ}
For any $p\in [\delta/2,2-\delta/2]$, $\tau>0$ and $\tau_2>\tau_1>0$, we have 
\begin{align*}
E^p[P](\tau)\leq& C\bar{E}[P]^{\leq 2} \tau^{p-2+\delta/2},\\
\int_{\tau_1}^{\tau_2} E^{p-1}[P](\tau)d\tau\leq& C\bar{E}[P]^{\leq 3} \tau_1^{p-2+\delta/2}.
\end{align*}
\begin{align*}
E^p[Q](\tau)\leq& C\bar{E}[Q]^{\leq 2} \tau^{p-2+\delta/2},\\
\int_{\tau_1}^{\tau_2} E^{p-1}[Q](\tau)d\tau\leq& C\bar{E}[Q]^{\leq 3} \tau_1^{p-2+\delta/2}.
\end{align*}
\end{corollary}
\begin{proof}
By applying (\ref{P hierarchy}) with $p=2-\delta/2$ in $[0,\tau]$, we get
\begin{align*}
E^{2-\delta/2}[P](\tau)\leq CE^{2-\delta/2}[P](0).
\end{align*}
From (\ref{P hierarchy}) with $p=2-\delta/2$ and $\tau_2=2^{k+1}$ and $\tau_1=2^k$, we can find $\bar{\tau}_k\in [2^k,2^{k+1}]$ such that
\begin{align*}
E^{1-\delta/2}[P](\bar{\tau}_k)\leq \frac{C}{\bar{\tau}_k}\bar{E}[P]^{\leq 1}.
\end{align*}
Then $E^{1-\delta/2}[P](\tau)\leq \frac{C}{\tau}\bar{E}[P]^{\leq 1}$ follows from applying (\ref{P hierarchy}) with $p=1-\delta/2$. By the interpolation, we have for all $p\in [1-\delta/2,2-\delta/2]$,
\begin{align*}
E^p[P](\tau)\leq& C\bar{E}[P]^{\leq 1} \tau^{p-2+\delta/2}.
\end{align*}
From (\ref{P hierarchy}) with $p=1+\delta/2$, $\tau_2=2^{k+1}$ and $\tau_1=2^k$, we can find $\bar{\tau}_k\in [2^k,2^{k+1}]$ such that
\begin{align*}
E^{\delta/2}[P](\bar{\tau}_k)\leq \frac{C}{\bar{\tau}_k}E^{1+\delta/2,\leq 1}[P](2^k)\leq C \bar{E}[P]^{\leq 2} \bar{\tau}_k^{-2+\delta}.
\end{align*}
Applying (\ref{P hierarchy}) one more time with $p=\delta/2$ the result follows. The proof for $Q$ is the same.
\end{proof}
\begin{theorem}\label{thm_H0}
We have for any $p\in [\delta,2-\delta]$
\begin{align}
E^p[H_0]\leq C\Big( \bar{E}[H_0]^{\leq 2+m}+\bar{E}[P]^{\leq 5+m} \Big)\tau^{-2+p+\delta},
\end{align}
and
\begin{align}
\int_{\tau_1}^{\tau_2}E^{p-1}[H_0](\tau)d\tau\leq C\Big( \bar{E}[H_0]^{\leq 3+m}+\bar{E}[P]^{\leq 6+m} \Big)\tau_1^{-2+p+\delta},
\end{align}
where $m$ is the integer defined by $m(\delta):=\left \lceil -\log_2\delta \right \rceil+1$.
\end{theorem}
\begin{proof}
From (\ref{H_0 hierarchy}) with $p=2-\delta$,
\begin{align}
E^{2-\delta}[H_0](\tau)+\int_0^{\tau}E^{1-\delta,deg}[H_0](s)ds\leq C\left(E^{2-\delta}[H_0](0)+\int_{0}^\tau \sqrt{E^{2-\delta}[H_0](s)}\sqrt{E^{\delta/2-1}[P](s)}ds\right).
\end{align}
Denote the right hand side by $F(\tau)$. Then $F$ satisfies an ordinary differential inequality
\begin{align*}
\frac{d}{d\tau} F(\tau)\leq C \sqrt{F(\tau)}\sqrt{E^{\delta/2-1}[P](\tau)}.
\end{align*}
Because
\begin{align*}
\int_1^\infty \sqrt{E^{\delta/2-1}[P](s)} ds&=\sum_{k=0}^{\infty}\int_{2^k}^{2^{k+1}}\sqrt{E^{\delta/2-1}[P](s)} ds\\
                                          &\leq\sum_{k=0}^{\infty} \left(\int_{2^k}^{2^{k+1}}E^{\delta/2-1}[P](s) ds \right)^{1/2}(2^k)^{1/2}\\
                                          &\leq \sum_{k=0}^{\infty} C\sqrt{\bar{E}[P]^{\leq 3}}(2^k)^{\frac{\delta-1}{2}}\leq C\sqrt{\bar{E}[P]^{\leq 3}},
\end{align*}
we obtain 
\begin{align*}
\sqrt{F(\tau)}\leq \sqrt{F(0)}+C\sqrt{\bar{E}[P]^{\leq 3}}).
\end{align*}
Therefore,
\begin{align}
E^{2-\delta}[H_0](\tau)\leq C\Big( \bar{E}[H_0]+\bar{E}[P]^{\leq 3} \Big),\\
\int_{\tau_1}^{\tau_2}E^{1-\delta}[H_0](\tau)d\tau \leq C\Big( \bar{E}[H_0]^{\leq 1}+\bar{E}[P]^{\leq 4} \Big).
\end{align}
From mean value theorem, there exists $\bar{\tau}_k\in [2^k,2^{k+1}]$ such that
\begin{align*}
E^{1-\delta}[H_0](\bar{\tau}_k)\leq C(\bar{E}[H_0]^{\leq 1}+\bar{E}[P]^{\leq 4})\bar{\tau}_k^{-1}.
\end{align*}
For any $\tau\in (\bar{\tau}_k,\bar{\tau}_{k+1})$, we have
\begin{align*}
&E^{1-\delta}[H_0](\tau)+\int_{\bar{\tau}_k}^\tau E^{-\delta,deg}[H_0](s)ds\\
\leq &C\left( E^{1-\delta}[H_0](\bar{\tau}_k)+\int_{\bar{\tau}_k}^\tau \sqrt{E^{1-\delta}[H_0](s)}\sqrt{E^{\delta/2-1}[P](s)}ds \right).
\end{align*}
Again denote the right hand side by $F(\tau)$. We have the same inequality 
\begin{align*}
\frac{d}{d\tau} F(\tau)\leq C \sqrt{F(\tau)}\sqrt{E^{\delta/2-1}[P](\tau)}.
\end{align*}
Integrating from $\bar{\tau}_k$ to $\tau$, we get
\begin{align*}
\sqrt{F(\tau)}&\leq \sqrt{F(\bar{\tau}_k)}+C\sqrt{\bar{E}[P]^{\leq 3}}\tau^{\frac{\delta-1}{2}}\\
              &\leq C(\sqrt{\bar{E}[H_0]^{\leq 1}}+\sqrt{\bar{E}[P]^{\leq 4}})\tau^{\frac{\delta-1}{2}}.
\end{align*}
Therefore
\begin{align}
E^{1-\delta}[H_0](\tau)\leq  C({\bar{E}[H_0]^{\leq 1}}+{\bar{E}[P]^{\leq 4}})\tau^{-1+\delta},
\end{align}
and
\begin{align}
\int_{\tau_1}^{\tau_2} E^{-\delta}[H_0](\tau)d\tau \leq C({\bar{E}[H_0]^{\leq 2}}+{\bar{E}[P]^{\leq 5}})\tau_1^{-1+\delta}.
\end{align}
To get better decay, we first note that by Cauchy-Schwarz, (\ref{H_0 hierarchy}) implies
\begin{equation}\label{linear estimate photon}
\begin{split}
&E^p[H_0](\tau_2)+\int_{\tau_1}^{\tau_2} E^{p-1,\deg}[H_0](\tau) d\tau\\
\leq & C \left( E^p[H_0](\tau_1)+\int_{\tau_1}^{\tau_2} \int_{\Sigma_\tau} r^{p-5}|P|^2 dVol_\tau d\tau +\int_{\tau_1}^{\tau_2} \int_{\Sigma_{\tau}\cap \{2M\leq r \leq 4M\}} \sqrt{e [H_0]}|P| dVol_\tau d\tau \right).
\end{split}
\end{equation}
Replacing (\ref{H_0 hierarchy}) by (\ref{linear estimate photon}) in estimating $E^{1-\delta}[H_0]$, we get
\begin{align*}
&E^{1-\delta}[H_0](\tau)+\int_{\bar{\tau}_k}^\tau E^{-\delta,deg}[H_0](s)ds \\
\leq & C\left( E^{1-\delta}[H_0](\bar{\tau}_k)+\int_{\bar{\tau}_k}^\tau E^{\delta/2-1}[P](s)ds+\left(\int_{\bar{\tau}_k}^\tau E^{-\delta}[H_0](s)ds\right)^{1/2}\left(\int_{\bar{\tau}_k}^\tau E^{\delta/2-1}[P](s)ds\right)^{1/2} \right)\\
 \leq & C\left( (\bar{E}[H_0]^{\leq 1}+\bar{E}[P]^{\leq 4})\bar{\tau}_k^{-1}+\bar{E}[P]^{\leq 3}\bar{\tau}_{k}^{-2+\delta}+(\bar{E}[H_0]^{\leq 2}+\bar{E}[P]^{\leq 5})\tau^{-3/2+\delta} \right)\\
 \leq & C(\bar{E}[H_0]^{\leq 2}+\bar{E}[P]^{\leq 5})\tau^{-1}.
\end{align*}
From the interpolation and the mean value theorem, we can again obtain $\bar{\tau}_k\in [2^k,2^{k+1}]$ such that
\begin{align*}
E^{\delta}[H_0](\bar{\tau}_k)\leq C(\bar{E}[H_0]^{\leq 3}+\bar{E}[P]^{\leq 6})\bar{\tau}_k^{-2+2\delta}.
\end{align*}
For any $\tau\in (\bar{\tau}_k,\bar{\tau}_{k+1})$, we applying (\ref{linear estimate photon}) with $p=\delta$ in $[\bar{\tau}_k,\tau]$ to obtain
\begin{align*}
&E^{\delta}[H_0](\tau)+\int_{\bar{\tau}_k}^\tau E^{-1+\delta,deg}[H_0](s)ds \\
\leq &C\left( E^\delta[H_0](\bar{\tau}_k)+\int_{\bar{\tau}_k}^\tau E^{\delta/2-1}[P](s)ds+\left( \int_{\bar{\tau}_k}^\tau E^{-\delta}[H_0](s)ds \right)^{1/2}\left( \int_{\bar{\tau}_k}^\tau E^{\delta/2-1}[P](s)ds \right)^{1/2} \right)\\
\leq &C\Big( \bar{E}[H_0]^{\leq 3}+\bar{E}[P]^{\leq 6} \Big)\tau^{-3/2+\delta/2}.
\end{align*}
Once we get decay of $e[H_0]$ near $r=3M$, we can use it to obtain better decay at the expense of loss of derivatives. Systematically, we have
\begin{equation}
\begin{split}
&E^{\delta}[H_0](\tau)+\int_{\bar{\tau}_k}^\tau E^{-1+\delta,deg}[H_0](s)ds \\
\leq &C\Big( \bar{E}[H_0]^{\leq 2+j}+\bar{E}[P]^{\leq 5+j} \Big)\tau^{\lambda_j},
\end{split}
\end{equation}
where $\lambda_j=\max\left\{-2+2\delta,\left(-1\right) \left(\frac{1}{2^j}\right)+\left( -2+\delta \right)\left(\frac{2^j-1}{2^j}\right)\right\}$. The number $m$ is chosen such that $\lambda_m=-2+2\delta$.

\end{proof}

\section{Estimate for $H_1$ and $H_2$}
In this section, we prove the decay estimate for $H_1$ and $H_2$. The proof is similar to the one in the previous section. In the Minkowski spacetime where $M=0$, (\ref{eqiation_H1H2}) is diagonalizable since it corresponds to scalar linear wave equations. For $M>0$, we use (\ref{wavegauge}), (\ref{def_P}) and (\ref{def_Q}) to control the off-diagonal term by $P,Q$ and $H_0$. Through this section we assume $H_1$ and $H_2$ are supported on the harmonic number $\ell\geq 2$.
\subsection{$T$ vector field}
Consider
\begin{align*}
J_1[H_1]_a&:=\Big(T_{ab}[H_1]-\frac{1}{2}V_1|H_1|^2g_{ab}\Big)T^b,\\
J_1[H_2]_a&:=\Big(T_{ab}[H_2]-\frac{1}{2}V_2|H_2|^2g_{ab}\Big)T^b.
\end{align*}
Then
\begin{align*}
\Div J_1[H_1]&=\frac{1}{r^2}\left(1-\frac{3M}{r}\right)D^\dagger H_2\cdot\nabla_T H_1,\\
\Div J_1[H_2]&=\frac{2}{r^2}D H_1\cdot\nabla_T H_2.
\end{align*}
Note that
\begin{align*}
&\nabla^a\Big(\left(-\frac{2}{r^2} D H_1\cdot H_2 g_{ab} \right)T^b\Big)\\
=&-\frac{2}{r^2}D \nabla_T H_1\cdot H_2-\frac{2}{r^2}DH_1\cdot\nabla_T H_2\\
=_s&-\frac{2}{r^2}\nabla_T H_1\cdot D^\dagger H_2-\frac{2}{r^2}DH_1\cdot\nabla_T H_2.
\end{align*}
Let $h_1(r)$ be a non-increasing function in $r$, $h_1(r)= 5$ in $[2M,4M]$ and $h_1(r)\geq 1$. The precise requirement of $h_1(r)$ will be determined in subsection 4.5. We define
\begin{align}
Q_{ab}&:=2\Big(T_{ab}[H_1]-\frac{1}{2}V_1|H_1|^2g_{ab}\Big)+\Big(T_{ab}[H_2]-\frac{1}{2}V_2|H_2|^2g_{ab}\Big)-\left(\frac{2}{r^2} D H_1\cdot H_2 g_{ab} \right),
\end{align}
\begin{align}
J_1[H_1,H_2]_a&:=Q_{ab}T^b+2(h_1(r)-1)J_1[H_1]_a,\\
K_1[H_1,H_2]&:=\Div J_1[H_1,H_2],\\
e_1[H_1,H_2]&:=J_1[H_1,H_2]\cdot n_{\Sigma_\tau}.
\end{align}
Then
\begin{align*}
K_1[H_1,H_2]&=_s\mathring{K}_1[H_1,H_2]+Err_1[H_1,H_2],\\
\mathring{K}_1[H_1,H_2]&=\left(-\frac{1}{2}h_1'(r)\right)\left( |\nabla_{\underline{L}}H_1|^2-|\nabla_L H_1|^2 \right),\\
Err_1[H_1,H_2]&=\left( -\frac{6M}{r^3}+\frac{2(h_1(r)-1)}{r^2}\left(1-\frac{3M}{r} \right) \right) D^\dagger H_2\cdot \nabla_T H_1.
\end{align*}
Now we show that $e_1[H_1,H_2]$ is comparable to $e^T[H_1]+e^T[H_2]$ when $H_1$ and $H_2$ are supported on $\ell\geq 2$. Let $f_1:=-T\cdot n_{\Sigma_\tau}\geq 0$. Then
\begin{align*}
e_1[H_1,H_2]\geq f_1\left( h_1|\nablas H_1|^2+\frac{1}{2}|\nablas H_2|^2+h_1V_1 |H_1|^2+\frac{1}{2}V_2 |H_2|^2+\frac{2}{r^2}DH_1\cdot H_2 \right).
\end{align*}
As $H_1$ and $H_2$ are supported on the harmonic number $\ell$,
\begin{align*}
e_1[H_1,H_2]\geq_s &\frac{f_1}{r^2}\left( h_1\left(\ell(\ell+1)-1+5-\frac{18M}{r}\right)|H_1|^2+\frac{1}{2}\Big( \ell(\ell+1)-4+2 \Big)|H_2|^2\right.\\
&\left.-2\Big(2\ell(\ell+1)-4\Big)^{1/2}|H_1||H_2| \right).
\end{align*}
From $h_1\geq 1$, the above quadratic form is non-negative in $[3M,\infty)$ and strictly positive in $[4M,\infty)$. In fact, it's strictly positive in $[2M,\infty)$ if $\ell\geq 3$. From $h_1=5$ in $[2M,4M]$, the strict positivity also holds in $[2M,4M]$ for $\ell=2$.
\subsection{red-shift vector field}
Consider a vector $Y$ such that
\begin{align*}
Y\Big|_{r=2M}&=-2\frac{\partial}{\partial R},\\
\nabla_Y Y\Big|_{r=2M}&=-sT-\sigma Y,
\end{align*}
for some $s,\sigma>0$ to be determined. We define
\begin{align}
J_2[H_1,H_2]_a&:=\Big(T_{ab}[H_1]+T_{ab}[H_2]\Big)Y^b,\\
K_2[H_1,H_2]&:=\Div J_2[H_1,H_2],\\
e_2[H_1,H_2]&:=J_2[H_1,H_2]\cdot n_{\Sigma_\tau}.
\end{align}
From (\ref{g_redshift}), we have at $r=2M$,
\begin{align*}
K_2[H_1,H_2]=&\left(\frac{s}{2}|\nabla_v H_1|^2+\frac{1}{2M}|\nabla_R H_1|^2+\frac{2}{M}\nabla_v H_1\nabla_R H_1+\frac{\sigma}{2}|\nablas H_1|^2\right)\\
            +&\left(\frac{s}{2}|\nabla_v H_2|^2+\frac{1}{2M}|\nabla_R H_2|^2+\frac{2}{M}\nabla_v H_2\nabla_R H_2+\frac{\sigma}{2}|\nablas H_2|^2\right)\\
            +&\left(V_1H_1+\frac{1}{r^2}\left(1-\frac{3M}{r}\right)D^\dagger H_2\right)\cdot\nabla_Y H_1+\left(V_2H_2+\frac{2}{r^2}DH_1\right)\cdot\nabla_Y H_2.
\end{align*}
By choosing $s=\sigma$ large enough, we have $K_2[H_1,H_2]\gtrsim_s e[H_1]+e[H_2]$ on the horizon. We don't need to introduce the error term for the red-shift current, but for consistency of notation, we still denote
\begin{align*}
K_2[H_1,H_2]&=\mathring{K}_2[H_1,H_2]+Err_2[H_1,H_2],\\
Err_2[H_1,H_2]&=0.
\end{align*}
 Extending $Y$ smoothly to the exterior as in subsection \ref{sr_H0}, we obtain
\begin{lemma}
\begin{equation*}
\begin{array}{cc}
\mathring{K}_2[H_1,H_2] \gtrsim_s e[H_1]+e[H_2] & \textup{in}\ [2M,r_0],\\
\Big|\mathring{K}_2[H_1,H_2]\Big| \lesssim_s e[H_1]+e[H_2] & \textup{in}\ [r_0,r_1],\\
\mathring{K}_2[H_1,H_2] = 0 & \textup{in}\ [r_1,\infty).
\end{array}
\end{equation*}
Furthermore, $e_2[H_1,H_2]\geq_s 0$ and
\begin{equation*}
\begin{split}
e_2[H_1,H_2]+e_1[H_1,H_2]&\gtrsim_s e[H_1]+e[H_2].
\end{split}
\end{equation*}
\end{lemma}
\subsection{Morawetz vector field}
For any radial function $f(r)$, let $X=f\D\frac{\partial}{\partial r}$ and $\omega^X=\D\left(\frac{\partial f}{\partial r}+\frac{2f}{r}\right)$. From (\ref{g_Morawetz}),
\begin{align*}
&\nabla^a\left(T_{ab}[H_1]X^b+\frac{1}{4}\omega^X \nabla_a |H_1|^2-\frac{1}{4}\nabla_a\omega^X |H_1|^2-\frac{1}{2}V_1|H_1|^2 X_a\right)\\
=&\D^2\frac{\partial f}{\partial r}|\nabla_r H_1|^2+\frac{f}{r}\left(1-\frac{3M}{r}\right)|\nablas H_1|^2+\left(-\frac{1}{4}\Box\omega^X-\frac{1}{2}\nabla_X V_1-\frac{M}{r^2} fV_1\right)|H_1|^2\\
&+\frac{1}{r^2}\left(1-\frac{3M}{r}\right)D^\dagger H_2\cdot\left( \nabla_X H_1+\frac{1}{2}\omega^X H_1 \right).
\end{align*}
Similarly,
\begin{align*}
&\nabla^a\left(T_{ab}[H_2]X^b+\frac{1}{4}\omega^X \nabla_a |H_2|^2-\frac{1}{4}\nabla_a\omega^X |H_2|^2-\frac{1}{2}V_2|H_2|^2 X_a\right)\\
=&\D^2\frac{\partial f}{\partial r}|\nabla_r H_2|^2+\frac{f}{r}\left(1-\frac{3M}{r}\right)|\nablas H_2|^2+\left(-\frac{1}{4}\Box\omega^X-\frac{1}{2}\nabla_X V_2-\frac{M}{r^2} fV_1\right)|H_2|^2\\
&+\frac{2}{r^2}D H_1\cdot\left( \nabla_X H_2+\frac{1}{2}\omega^X H_2 \right).
\end{align*}
Also,
\begin{align*}
&\nabla^a\left(-\frac{2}{r^2}DH_1\cdot H_2 X_a\right)\\
=_s&-\frac{2}{r^2}\nabla_X H_1\cdot D^\dagger H_2-\frac{2}{r^2}DH_1\cdot\nabla_X H_2-\nabla_X\left(\frac{2}{r^2}\right)DH_1\cdot H_2\\
&-\frac{1}{r^2}\left(\omega^X+\frac{2M}{r^2}f\right)DH_1\cdot H_2-\frac{1}{r^2}\left(\omega^X+\frac{2M}{r^2}f\right)H_1\cdot D^\dagger H_2.
\end{align*}
We define,
\begin{equation}
\begin{split}
\tilde{J}_3[H_1,H_2]_a:=&2\left(T_{ab}[H_1]X^b+\frac{1}{4}\omega^X \nabla_a |H_1|^2-\frac{1}{4}\nabla_a\omega^X |H_1|^2-\frac{1}{2}V_1|H_1|^2 X_a\right)\\
                        &+\left(T_{ab}[H_2]X^b+\frac{1}{4}\omega^X \nabla_a |H_2|^2-\frac{1}{4}\nabla_a\omega^X |H_2|^2-\frac{1}{2}V_2|H_2|^2 X_a\right)\\
                        &-\frac{2}{r^2}DH_1\cdot H_2 X_a.
\end{split}
\end{equation}
Then
\begin{equation}\label{H_12_Morawetz_far}
\begin{split}
&\Div \tilde{J}_3[H_1,H_2]\\
=_s &2\D^2\frac{\partial f}{\partial r}|\nabla_r H_1|^2+\frac{2f}{r}\left(1-\frac{3M}{r}\right)|\nablas H_1|^2+2\left(-\frac{1}{4}\Box\omega^X-\frac{1}{2}\nabla_X V_1-\frac{M}{r^2} fV_1\right)|H_1|^2\\
                             +&\D^2\frac{\partial f}{\partial r}|\nabla_r H_2|^2+\frac{f}{r}\left(1-\frac{3M}{r}\right)|\nablas H_2|^2+\left(-\frac{1}{4}\Box\omega^X-\frac{1}{2}\nabla_X V_2-\frac{M}{r^2} fV_1\right)|H_2|^2\\
                             &-\frac{6M}{r^3}D^\dagger H_2\cdot\nabla_X H_1+\left( -\nabla_X\left(\frac{2}{r^2}\right)-\frac{4M}{r^4}f-\frac{3M}{r^3}\omega^X \right) DH_1\cdot H_2.
\end{split}                             
\end{equation}
Let
\begin{align}
Err_{3,far}[H_1,H_2]&:=-\frac{6M}{r^3}D^\dagger H_2\cdot\nabla_X H_1.
\end{align}
We note that as $f\approx 1$ and $\omega^X\approx \frac{2}{r}$, $\Div \tilde{J}_3[H_1,H_2]-Err_{3,far}[H_1,H_2]$ is positive after integrating on $S^2$ for $r$ large enough. The reason is that the leading terms correspond the Morawetz estimate in the Minkowski spacetime. The main goal here is to make the bulk term positive in compact $r$ region up to quadratic terms involving $P,Q$ and $H_0$. To achieve this, we define the Morawetz current in $[2M,3M]$ and in $[3M,\infty)$ separately. We first consider the region $[3M,\infty)$. Denote by 
\begin{align*}
B_1&:=\left(-\frac{1}{4}\Box\omega^X-\frac{1}{2}\nabla_X V_1-\frac{M}{r^2} fV_1\right),\\
B_2&:=\left(-\frac{1}{4}\Box\omega^X-\frac{1}{2}\nabla_X V_2-\frac{M}{r^2} fV_2\right),
\end{align*}
and
\begin{align*}
E_1:=&-\frac{6M}{r^3}D^\dagger H_2\cdot\nabla_X H_1+\left( -\nabla_X\left(\frac{2}{r^2}\right)-\frac{4M}{r^4}f-\frac{3M}{r^3}\omega^X \right) DH_1\cdot H_2\\
    =_s&-\frac{6M}{r^3}\D f D^\dagger H_2\cdot\nabla_r H_1+\frac{f}{r^3}\left(4-\frac{18M}{r}+\frac{12M^2}{r^2}\right)DH_1\cdot H_2 \\
    &-\frac{3M}{r^3}\D \frac{\partial f}{\partial r} H_1\cdot D^\dagger H_2.
\end{align*}
By using (\ref{wavegauge}) to replace $D^\dagger H_2$, we have
\begin{align*}
E_1=_s &\frac{f}{r^3}\left(4-\frac{18M}{r}+\frac{12M^2}{r^2}\right) DH_1\cdot H_2-\frac{12M}{r^2}f\nabla_T H_0\cdot\nabla_r H_1-\frac{6M}{r^2}\frac{\partial f}{\partial r}\nabla_T H_0\cdot H_1\\
   &+\frac{12M}{r^2}\D f|\nabla_r H_1|^2+\frac{36M}{r^3}\D f H_1\cdot\nabla_r H_1\\
   &+\frac{6M}{r^2}\D \frac{\partial f}{\partial r}H_1\cdot\nabla_r H_1+\frac{18M}{r^3}\D\frac{\partial f}{\partial r}|H_1|^2.
\end{align*}
By Cauchy-Schwarz, 
\begin{align*}
\frac{6M}{r^2}\D \frac{\partial f}{\partial r}H_1\cdot\nabla_r H_1+\frac{18M}{r^3}\D\frac{\partial f}{\partial r}|H_1|^2\geq -\frac{M}{2r}\D\frac{\partial f}{\partial r}|\nabla_r H_1|^2.
\end{align*}
Thus
\begin{align*}
\Div \tilde{J}_3[H_1,H_2]\geq_s &\left(2\D^2\frac{\partial f}{\partial r}-\frac{M}{2r}\D\frac{\partial f}{\partial r}+\frac{12M}{r^2}\D f\right)|\nabla_r H_1|^2\\
&+\frac{2f}{r}\left(1-\frac{3M}{r}\right)|\nablas H_1|^2+2B_1|H_1|^2\\
                             &+\D^2\frac{\partial f}{\partial r}|\nabla_r H_2|^2+\frac{f}{r}\left(1-\frac{3M}{r}\right)|\nablas H_2|^2+B_2|H_2|^2\\
                             &+\frac{f}{r^3}\left(4-\frac{18M}{r}+\frac{12M^2}{r^2}\right) DH_1\cdot H_2-\frac{12M}{r^2}f\nabla_T H_0\cdot\nabla_r H_1-\frac{6M}{r^2}\frac{\partial f}{\partial r}\nabla_T H_0\cdot H_1\\
                             &+\frac{36M}{r^3}\D fH_1\cdot\nabla_r H_1.
\end{align*}
By Cauchy-Schwarz again,
\begin{align*}
\Div \tilde{J}_3[H_1,H_2]\geq_s &\frac{2f}{r}\left(1-\frac{3M}{r}\right)|\nablas H_1|^2+(2B_1-A_1)|H_1|^2\\
                               &+\D^2\frac{\partial f}{\partial r}|\nabla_r H_2|^2+\frac{f}{r}\left(1-\frac{3M}{r}\right)|\nablas H_2|^2+B_2|H_2|^2\\
                               &+\frac{f}{r^3}\left(4-\frac{18M}{r}+\frac{12M^2}{r^2}\right) DH_1\cdot H_2-\frac{12M}{r^2}f\nabla_T H_0\cdot\nabla_r H_1-\frac{6M}{r^2}\frac{\partial f}{\partial r}\nabla_T H_0\cdot H_1 ,
\end{align*}
where
\begin{align*}
A_1=\left(\frac{36M}{r^3}\D f\right)^2\bigg/4\left(2\D^2\frac{\partial f}{\partial r}-\frac{M}{2r}\D\frac{\partial f}{\partial r}+\frac{12M}{r^2}\D f\right).
\end{align*}
Suppose $H_1$ and $H_2$ are supported on the harmonic number $\ell$. Then
\begin{align*}
&\Div \tilde{J}_3[H_1,H_2]-\left(-\frac{12M}{r^2}f\nabla_T H_0\cdot\nabla_r H_1-\frac{6M}{r^2}\frac{\partial f}{\partial r}\nabla_T H_0\cdot H_1\right)\\
\geq_s&\left(2B_1+\frac{2f}{r^3}\left(1-\frac{3M}{r}\right)(\ell(\ell+1)-1)-A_1\right)|H_1|^2\\
                               &+\left( B_2+\frac{f}{r^3}\left(1-\frac{3M}{r}\right)(\ell(\ell+1)-4) \right)|H_2|^2\\
                               &-\frac{f}{r^3}\left| 4-\frac{18M}{r}+\frac{12M^2}{r^2} \right|\sqrt{2\ell(\ell+1)-4}|H_1||H_2|.
\end{align*}
We choose $f=\left(1+\frac{3M}{r}+\frac{2M^2}{r^2}\right)\left(1-\frac{3M}{r}\right)$ and denote by $\lambda=\ell (\ell+1)$. The discriminant of the above quadratic terms is
\begin{align*}
-D_{out}=&4\left(2B_1-A_1+\frac{2f}{r^3}\left(1-\frac{3M}{r}\right)(\lambda-1)\right)\left( B_2+\frac{f}{r^3}\left(1-\frac{3M}{r}\right)(\lambda-4) \right)\\
          &-\frac{f^2}{r^6}\left( 4-\frac{18M}{r}+\frac{12M^2}{r^2} \right)^2 (2\lambda-4).
\end{align*}
As $\lambda$ is large, $-D_{out}$ is about
\begin{align*}
&4\left(2B_1-A_1+\frac{2f}{r^3}\left(1-\frac{3M}{r}\right)\lambda\right)\left( B_2+\frac{f}{r^3}\left(1-\frac{3M}{r}\right)\lambda \right)\\	
          &-\frac{2f^2}{r^6}\left( 4-\frac{18M}{r}+\frac{12M^2}{r^2} \right)^2\lambda\\
        =&4\left((2B_1-A_1)B_2+\frac{2f^2}{r^6}\left(1-\frac{3M}{r}\right)^2\lambda^2\right)\\
         &+\frac{f}{r^3}\left(1-\frac{3M}{r}\right)\left( 4(2B_1+2B_2-A_1)-\frac{2f}{r^3\left(1-\frac{3M}{r}\right)}\left( 4-\frac{18M}{r}+\frac{12M^2}{r^2} \right)^2 \right)\lambda.
\end{align*}
By using Mathematica, we check that
\begin{align*}
(2B_1-A_1)B_2+\frac{2f^2}{r^6}\left(1-\frac{3M}{r}\right)^2\lambda^2
\end{align*}
and
\begin{align*}
4(2B_1+2B_2-A_1)-\frac{2f}{r^3\left(1-\frac{3M}{r}\right)}\left( 4-\frac{18M}{r}+\frac{12M^2}{r^2} \right)^2
\end{align*}
are positive in $[3M,3.3M]$ as $\ell\geq 3$. For $r\geq 3.3 M$, $-D_{out}$ is larger than
\begin{align*}
&\frac{2f^2}{r^6}\left(1-\frac{3M}{r}\right)^2\lambda^2-\frac{2f^2}{r^6}\left( 4-\frac{18M}{r}+\frac{12M^2}{r^2} \right)^2\lambda\\
=&\frac{2f^2}{r^6}\left(1-\frac{3M}{r}\right)^2\left(\lambda-\left( 4-\frac{18M}{r}+\frac{12M^2}{r^2} \right)^2\bigg/ \left(1-\frac{3M}{r}\right)^2\right)\lambda,
\end{align*}
which is positive as $\ell\geq 4$ since $\left| 4-\frac{18M}{r}+\frac{12M^2}{r^2} \right|\Big/ \left|1-\frac{3M}{r}\right|\leq 4$ in $[3.3M,\infty )$. From the above, we deduce that $-D_{out}>0$ in $[3M,\infty)$ for $\ell\geq 7$. We check $-D_{out}>0$ in $[3M,\infty)$ for the case $3\leq \ell\leq 6$ directly by using Mathematica. For $\ell=2$, we use that
\begin{align*}
\textup{div}\tilde{J}_3=_s&\left(2\D^2\frac{\partial f}{\partial r}+\frac{12M}{r}\D f\right)|\nabla_r H_1|^2\\
                        &+\left(\frac{10f}{r^3}\left(1-\frac{3M}{r}\right)+2B_1+\frac{18M}{r^3}\D\frac{\partial f}{\partial r} \right)|H_1|^2\\
                        &+\left(\frac{36M}{r^3}\D f+\frac{6M}{r^2}\D \frac{\partial f}{\partial r} \right)H_1\cdot\nabla_r H_1+\frac{f}{r^3}\left(4-\frac{18M}{r}+\frac{12M^2}{r^2}\right)DH_1\cdot H_2\\
                        &+\left( \frac{2f}{r^3}\left(1-\frac{3M}{r}\right)+B_2 \right)|H_2|^2+\D^2\frac{\partial f}{\partial r}|\nabla_r H_2|^2\\
                        &-\frac{12M}{r^2}f\nabla_T H_0\cdot \nabla_r H_1-\frac{6M}{r^2}\frac{\partial f}{\partial r}\nabla_T H_0\cdot H_1.
\end{align*}
The first three lines consist of a quadratic form involving $H_1,\ H_2$ and $\nabla_r H_1$ which can be checked to be positive definite in $[3M,\infty)$. \\\\
For $r\leq 3M$, we consider $\bar{X}=\left(1-\frac{3M}{r}\right)X$.
\begin{align*}
&\nabla^a\left(-\frac{2}{r^2}DH_1\cdot H_2 \bar{X}_a\right)\\
=&-\frac{2}{r^2}\left(1-\frac{3M}{r}\right)D\nabla_X H_1\cdot H_2-\frac{2}{r^2}\left(1-\frac{3M}{r}\right)DH_1\cdot\nabla_X H_2\\
                                                           &-\left(1-\frac{3M}{r}\right)\nabla_X\left(\frac{2}{r^2}\right)DH_1\cdot H_2-\frac{2}{r^2}\left(1-\frac{3M}{r}\right)\left(\omega^X+\frac{2M}{r^2}f\right)DH_1\cdot H_2\\
                                                           &-\frac{6M}{r^4}\D f DH_1\cdot H_2.
\end{align*}
\begin{equation}
\begin{split}
\bar{J}_3[H_1,H_2]_a:=&2\left(T_{ab}[H_1]X^b+\frac{1}{4}\omega^X \nabla_a |H_1|^2-\frac{1}{4}\nabla_a\omega^X |H_1|^2-\frac{1}{2}V_1|H_1|^2 X_a\right)\\
                        &+\left(T_{ab}[H_2]X^b+\frac{1}{4}\omega^X \nabla_a |H_2|^2-\frac{1}{4}\nabla_a\omega^X |H_2|^2-\frac{1}{2}V_2|H_2|^2 X_a\right)\\
                        &-\frac{2}{r^2}DH_1\cdot H_2 \bar{X}_a.
\end{split}
\end{equation}
Then
\begin{equation}
\begin{split}
&\Div \bar{J}_3[H_1,H_2]\\
=_s &2\D^2\frac{\partial f}{\partial r}|\nabla_r H_1|^2+\frac{2f}{r}\left(1-\frac{3M}{r}\right)|\nablas H_1|^2+2B_2|H_1|^2\\
                             &+\D^2\frac{\partial f}{\partial r}|\nabla_r H_2|^2+\frac{f}{r}\left(1-\frac{3M}{r}\right)|\nablas H_2|^2+B_1|H_2|^2\\
                             &+\frac{6M}{r^3}\D fDH_1\cdot\nabla_r H_2 +\frac{3M}{r^3}\D \frac{\partial f}{\partial r}DH_1\cdot H_2+\frac{4f}{r^3}\left(1-\frac{3M}{r}\right)^2 DH_1\cdot H_2.                        
\end{split}                             
\end{equation}
Denote the last line by $E_2$ and use (\ref{def_Q}) to replace $DH_1$,
\begin{align*}
E_2:=&\frac{6M}{r^3}\D fDH_1\cdot\nabla_r H_2 +\frac{3M}{r^3}\D \frac{\partial f}{\partial r}DH_1\cdot H_2+\frac{4f}{r^3}\left(1-\frac{3M}{r}\right)^2 DH_1\cdot H_2 \\
    =_s&-\frac{6M}{r^2}\D^2 f|\nabla_r H_2|^2\\
    &+\left(-\frac{3M}{r^2}\D^2\frac{\partial f}{\partial r}-\frac{4f}{r^2}\D\left(1-\frac{3M}{r}\right)^2\right)H_2\cdot\nabla_r H_2\\
    &+\frac{6M}{r^2}\D f\nabla_r H_2\cdot Q+\left(\frac{3M}{r^2}\D\frac{\partial f}{\partial r}+\frac{4f}{r^2}\left(1-\frac{3M}{r}\right)^2\right)H_2\cdot Q.
\end{align*}
Therefore for a fixed harmonic number $\ell$
\begin{align*}
&\Div \bar{J}_3[H_1,H_2]\\
=_s &\left(\frac{\partial f}{\partial r}-\frac{6M}{r^2}f\right)\D^2|\nabla_r H_2|^2+\left((\ell(\ell+1)-4)\frac{f}{r^3}\left(1-\frac{3M}{r}\right)+B_2\right)|H_2|^2\\
&+\left(-\frac{3M}{r^2}\D^2\frac{\partial f}{\partial r}-\frac{4f}{r^2}\D\left(1-\frac{3M}{r}\right)^2\right)H_2\cdot\nabla_r H_2\\
&+2\D^2\frac{\partial f}{\partial r}|\nabla_r H_1|^2+\left((\ell(\ell+1)-1)\frac{2f}{r^3}\left(1-\frac{3M}{r}\right)+2B_1\right)|H_1|^2\\
&+\frac{6M}{r^2}\D f\nabla_r H_2\cdot Q+\left(\frac{3M}{r^2}\D\frac{\partial f}{\partial r}+\frac{4f}{r^2}\left(1-\frac{3M}{r}\right)^2\right)H_2\cdot Q.
\end{align*}
The first two line consist a quadratic form involving $H_2$ and $\nabla_r H_2$ which increases in $\ell$. We check by Mathematica that it is positive definite for $\ell=2$ in $[2M,3M]$.
\\\\
Now we define $J_3[H_1,H_2]$ by
\begin{equation}
\begin{split}
J_3[H_1,H_2]:=&\hat{J}_3[H_1,H_2]-\frac{1}{2}(\epsilon_1g)\nabla_a(2|H_1|^2+|H_2|^2)+\frac{1}{2}\nabla_a(\epsilon_1 g)(2|H_1|^2+|H_2|^2),\\
\hat{J}_3[H_1,H_2]:=& \tilde{J}_3[H_1,H_2]\ \ \textup{in}\ \ [3M,\infty),\\
\hat{J}_3[H_1,H_2]:=& \bar{J}_3[H_1,H_2]\ \ \textup{in}\ \ [2M,3M].
\end{split}              
\end{equation}
Because $J_3[H_1,H_2]$ is continuous and piecewise smooth, we can still apply the divergence theorem. The bulk term and the boundary term are defined as before.
\begin{align}
K_3[H_1,H_2]:=&\Div J_3[H_1,H_2],\\
e_3[H_1,H_2]:=&J_3[H_1,H_2]\cdot n_{\Sigma_\tau}.
\end{align}
\begin{align*}
K_3[H_1,H_2]&=\mathring{K}_3[H_1,H_2]+Err_3[H_1,H_2]\\
            &=\mathring{K}_{3,far}[H_1,H_2]+Err_{3,far}[H_1,H_2].\\
\end{align*}
\begin{align*}
Err_{3,far}[H_1,H_2]&=-\frac{6M}{r^3}D^\dagger H_2\cdot\nabla_X H_1,\\
Err_3[H_1,H_2]&=\frac{6M}{r^2}\D f \nabla_rH_2\cdot Q\\
              &+\left(\frac{3M}{r^2}\D\frac{\partial f}{\partial r}+\frac{4f}{r^2}\left(1-\frac{3M}{r}\right)^2\right)H_2\cdot Q\ \ \textup{in}\ \ [2M,3M],\\
Err_3[H_1,H_2]&=-\frac{12M}{r^2}f\nabla_T H_0\cdot\nabla_r H_1-\frac{6M}{r^2}\frac{\partial f}{\partial r}\nabla_T H_0\cdot H_1 \ \ \textup{in}\ \ [3M,\infty).
\end{align*}
\begin{lemma} For $\epsilon_1>0$ small enough, we have $\mathring{K}_3[H_1,H_2]\geq_s 0$. Moreover, in $[r_0,\infty)$,
\begin{align*}
\mathring{K}_3[H_1,H_2]\gtrsim_s &\frac{1}{r^3}|\nabla_r H_1|^2+\frac{1}{r^3}\left(1-\frac{3M}{r}\right)^2|\nabla_t H_1|^2+\frac{1}{r}\left(1-\frac{3M}{r}\right)^2|\nablas H_1|^2+\frac{1}{r^3}|H_1|^2\\
                 &+\frac{1}{r^3}|\nabla_r H_2|^2+\frac{1}{r^3}\left(1-\frac{3M}{r}\right)^2|\nabla_t H_2|^2+\frac{1}{r}\left(1-\frac{3M}{r}\right)^2|\nablas H_2|^2+\frac{1}{r^3}|H_2|^2.
\end{align*}
For any $R_2$ large enough, we have in $[R_2,\infty)$,
\begin{align*}
\mathring{K}_{3,far}[H_1,H_2]\gtrsim_s &\frac{1}{r^3}|\nabla_r H_1|^2+\frac{1}{r^3}\left(1-\frac{3M}{r}\right)^2|\nabla_t H_1|^2+\frac{1}{r}\left(1-\frac{3M}{r}\right)^2|\nablas H_1|^2+\frac{1}{r^3}|H_1|^2\\
                 &+\frac{1}{r^3}|\nabla_r H_2|^2+\frac{1}{r^3}\left(1-\frac{3M}{r}\right)^2|\nabla_t H_2|^2+\frac{1}{r}\left(1-\frac{3M}{r}\right)^2|\nablas H_2|^2+\frac{1}{r^3}|H_2|^2.
\end{align*}
Furthermore, for bounded $r$ the error term is bounded by
\begin{align*}                 
\Big(Err_3[H_1,H_2]\Big)^2&\lesssim \Big(|\nabla_r H_1|^2+|\nabla_r H_2|^2+|H_1|^2+|H_2|^2\Big)\Big(e[H_0]+e[Q]\Big).
\end{align*}
\end{lemma}
\begin{proof}
As $\epsilon_1=0$, the result follows the above discussion except the absence of $|\nabla_t H_1|^2$ and $|\nabla_t H_2|^2$ terms. By choosing $g(r)=\frac{1}{r^3}\D\left(1-\frac{3M}{r}\right)^2\geq 0$ and $\epsilon_1>0$  small enough, we can add a term comparable to $\frac{1}{r^3}(|\nabla_t H_1|^2+|\nabla_t H_2|^2)$ to $\mathring{K}_3$ and $\mathring{K}_{3,far}$ without disturbing the positivity of other terms.
\end{proof}
By repeating the proof of Lemma \ref{H0_Morawetz_boundary}, we have
\begin{lemma}
There exists $C_0\geq 1$ such that we have the following estimates: 
\begin{align*}
2e^X[H_1]+e^X[H_2]\geq_s -C_0 e_1[H_1,H_2],
\end{align*}
For any $R_1$ large enough,
\begin{align*}
\left|\int_{\Sigma_\tau} e_3[H_1,H_2]-2e^X[H_1]-e^X[H_2] dVol_{\Sigma_\tau}\right|\leq &C_0\int_{\Sigma_\tau\cap\{2M\leq r\leq R_1\}} e_1[H_1,H_2]dVol_{\Sigma_\tau}\\
                                                                    +&C_0\int_{\Sigma_\tau\cap\{R_1\leq r\}} e^L[H_1]+e^L[H_2]dVol_{\Sigma_\tau}.
\end{align*}
\end{lemma}
\subsection{$r^p$-hierarchy}
The $r^p$-hierarchy deals with the region where $r$ is large and the contribution of $M$ is negligible. Therefore in principle the estimate in this subsection follows the ordinary $r^p$-hierarchy in the Minkowski spacetime. Define
\begin{equation}
\begin{split}
&\tilde{J}^p_4[H_1]_a\\
:=&\frac{r^p}{r^2\D} T_{ab}[rH_1]L^b-\frac{Mr^p}{r^2\D}|rH_1|^2L_a-\frac{r^p}{2r^2\D}V_1|rH_1|^2L_a,
\end{split}
\end{equation}
\begin{equation}
\begin{split}
&\tilde{J}^p_4[H_2]_a\\
:=&\frac{r^p}{r^2\D} T_{ab}[rH_2]L^b-\frac{Mr^p}{r^2\D}|rH_2|^2L_a-\frac{r^p}{2r^2\D}V_2|rH_2|^2L_a.
\end{split}
\end{equation}
From (\ref{g_rp}),
\begin{align*}
&\Div \tilde{J}_4^p[H_1]\\
=&\left(1-\frac{2M}{r}\right)^{-2}r^{p-3}\left( \frac{p}{2}\left(1-\frac{2M}{r}\right)-\frac{M}{r} \right)|\nabla_L (rH_1)|^2+\left(1-\frac{p}{2}\right)r^{p-3}|\slashed{\nabla}(rH_1)|^2\\
                     &+\left( (3-p)Mr^{p-6}-\frac{r^p}{r^2\left(1-\frac{2M}{r}\right)}\frac{1}{r}\left(1-\frac{M}{r}\right)V_1-\left(1-\frac{2M}{r}\right)\partial_r\left( \frac{r^p}{2r^2(1-\mu)}V_1 \right) \right)|(rH_1)|^2\\
                     &+\frac{r^p}{r^2\D}\Big(\nabla_L (rH_1)\Big)\cdot r\left( \frac{1}{r^2}\left(1-\frac{3M}{r}\right)D^\dagger H_2 \right).
\end{align*}
\begin{align*}
&\Div \tilde{J}_4^p[H_2]\\
=&\left(1-\frac{2M}{r}\right)^{-2}r^{p-3}\left( \frac{p}{2}\left(1-\frac{2M}{r}\right)-\frac{M}{r} \right)|\nabla_L (rH_2)|^2+\left(1-\frac{p}{2}\right)r^{p-3}|\slashed{\nabla}(rH_2)|^2\\
                     &+\left( (3-p)Mr^{p-6}-\frac{r^p}{r^2\left(1-\frac{2M}{r}\right)}\frac{1}{r}\left(1-\frac{M}{r}\right)V_2-\left(1-\frac{2M}{r}\right)\partial_r\left( \frac{r^p}{2r^2(1-\mu)}V_2 \right) \right)|(rH_2)|^2\\
                     &+\frac{r^p}{r^2 \D}\Big(\nabla_L (rH_2)\Big)\cdot r\left( \frac{2}{r^2}D H_1 \right).
\end{align*}
Also,
\begin{align*}
&\nabla^a\left(\frac{r^p}{r^2\D}\frac{2}{r^2}(r DH_1)\cdot (rH_2)L_a\right)\\
=_s&2(p-2)r^{p-3} DH_1\cdot H_2+l.o.t.\\
&+\frac{r^p}{r^2\D}\frac{2}{r^2}\bigg(\nabla_L(rH_2)\cdot (rDH_1)+\nabla_L(rH_1)\cdot (rD^\dagger H_2)\bigg).
\end{align*}
We define
\begin{align}
\tilde{J}^p_4[H_1,H_2]_a:=2\tilde{J}^p_4[H_1]_a+\tilde{J}^p_4[H_2]_a-\frac{r^p}{r^2\D}\frac{2}{r^2}(rH_1)\cdot (rH_2)L_a.
\end{align}
Then the leading order terms in $\Div \tilde{J}^p_4[H_1,H_2]$ are
\begin{align*}
\Div \tilde{J}^p_4[H_1,H_2]=_s&pr^{p-3}|\nabla_L (rH_1)|^2+(2-p)r^{p-3}|\nablas (rH_1)|^2+5(2-p)r^{p-5}|rH_1|^2\\
                          +&\frac{p}{2}r^{p-3}|\nabla_L (rH_2)|^2+(1-\frac{p}{2})r^{p-3}|\nablas (rH_2)|^2+(2-p)r^{p-5}|rH_2|^2+\\
                          -&2(p-2)r^{p-3}DH_1\cdot H_2+\frac{r^p}{r^2\D}\left(-\frac{6M}{r^3}\right)\nabla_L (rH_1)\cdot (rD^\dagger H_2)+l.o.t.
\end{align*}
Therefore, as $r$ is large enough,
\begin{align*}
\Div \tilde{J}^p_4[H_1,H_2]\gtrsim_s &r^{p-1}|\nabla_L H_1|^2+r^{p-1}|\nablas H_1|^2+r^{p-3}|H_1|^2\\
                          +&r^{p-1}|\nabla_L H_2|^2+r^{p-1}|\nablas H_2|^2+r^{p-3}|H_2|^2.
\end{align*}
Consider
\begin{align*}
\nabla^a\left(r^{p-2}Q_{ab}T^b\right)=&-6Mr^{p-5}D^\dagger H_2 \nabla_T H_1\\
                                      &+\left(\frac{2-p}{4}\right)r^{p-3}\left(2|\nabla_{\underline{L}}H_1|^2-2|\nabla_L H_1|^2+|\nabla_{\underline{L}}H_2|^2-|\nabla_L H_2|^2\right).	
\end{align*}
Function $\eta(r)$ is chosen to be a cut-off function such that $\eta=1$ for $r\geq R_0$ and $\eta=0$ for $r\leq R_0-M$. We define
\begin{align}
J^p_4[H_1,H_2]_a&:=\eta\left( \tilde{J}^p_{4}[H_1,H_2]+r^{p-2}Q_{ab}T^b \right),\\
K^p_4[H_1,H_2]&:=\Div J^p_4[H_1,H_2],\\
e^p_4[H_1,H_2]&:=J^p_4[H_1,H_2]\cdot n_{\Sigma\tau}.
\end{align}
In this case $K^p_4[H_1,H_2]$ is positive for large $r\geq R_0$ and we don't need the error term. Nevertheless, for consistency of notation, we still denote
\begin{align*}
K^p_4[H_1,H_2]&=\mathring{K}^p_4[H_1,H_2]+Err^p_4[H_1,H_2],\\
Err^p_4[H_1,H_2]&=0.
\end{align*} 
In summary, we obtain
\begin{lemma} There exists a constant $C_1>0$ such that
\begin{align*}
\begin{array}{cc}
\mathring{K}_4^p[H_1,H_2]\geq \frac{1}{C_1} (e^{p-1}[H_1]+e^{p-1}[H_2]) & \textup{in}\ [R_0,\infty),\\
\Big|\mathring{K}_4^p[H_1,H_2]\Big|\lesssim e[H_1]+e[H_2] & \textup{in}\ [R_0-M,R_0],\\
\mathring{K}_4^p[H_1,H_2]=0 & \textup{in}\ [2M,R_0-M].
\end{array}
\end{align*}
Furthermore, $e^p_4[H_1,H_2]$ is non-negative and
\begin{align*}
\begin{array}{cc}
e^p_4[H_1,H_2]\approx e^p[H_1]+e^p[H_2] & \textup{in}\ [R_0,\infty).
\end{array}
\end{align*}
\end{lemma}
\subsection{combination}
We take $\epsilon_2>0$ small such that for all $p\in [\delta_1,2-\delta_2]$,
\begin{align*}
\begin{array}{cc}
\mathring{K}_3[H_1,H_2]\geq 2\epsilon_2 \Big|\mathring{K}^p_4[H_1,H_2]\Big| & \textup{in}\ [R_0-M,R_0]. 
\end{array}
\end{align*}
Then we pick $R_1\geq R_0$ such that 
\begin{align*}
\frac{30C_0}{\log 2}\leq \frac{1}{2C_1}\epsilon_2 R_1^{\delta_1},
\end{align*}
and
\begin{align*}
\begin{array}{cc}
\epsilon_2 e_4^{p=\delta_1}[H_1,H_2]\geq 2C_0 e^L[H_1,H_2] & \textup{in}\ [R_1,\infty).
\end{array}
\end{align*}
We consider a non-increasing function $h_2(r)$ satisfying
\begin{align*}
\begin{array}{cc}
h_2(r)= 3C_0 & \textup{in}\ [2M,R_1],\\
h_2(r)= 1 & \textup{in}\ [2R_1,\infty),\\
|h_2'(r)|\leq \frac{3C_0}{\log 2}r^{-1}& \textup{in}\ [R_1,2R_1].
\end{array}
\end{align*}
Also, we pick the function $h_1(r)$ by
\begin{align*}
\begin{array}{cc}
h_1(r)= 5 & \textup{in}\ [2M,R_1],\\
h_1(r)= 1 & \textup{in}\ [2R_1,\infty),\\
|h_1'(r)|\leq \frac{5}{\log 2}r^{-1}& \textup{in}\ [R_1,2R_1].
\end{array}
\end{align*}
The fifth current is defined by
\begin{align}
J_5[H_1,H_2]_a&:= h_2(r)J_1[H_1,H_2]_a,\\
K_5[H_1,H_2]&:=\Div J_5[H_1,H_2],\\
e_5[H_1,H_2]_a&:=J_5[H_1,H_2]\cdot n_{\Sigma_\tau}.
\end{align}
Also,
\begin{align*}
K_5[H_1,H_2]&=\mathring{K}_5[H_1,H_2]+Err_5[H_1,H_2],\\
\mathring{K}_5[H_1,H_2]&=\left(-\frac{1}{4}h'_2\right)\left(|\nabla_{\underline{L}}H_2|^2-|\nabla_L H_2|^2\right)+\left(-\frac{(h_1h_2)'}{2}\right)\left(|\nabla_{\underline{L}}H_1|^2-|\nabla_L H_1|^2\right),\\
Err_5[H_1,H_2]&=h_2\left(-\frac{6M}{r^3}+\frac{2(h_1-1)}{r^2}\left(1-\frac{3M}{r}\right)\right)D^\dagger H_2\cdot \nabla_T H_1.
\end{align*}
Then we have
\begin{lemma}
\begin{align*}
\begin{array}{cc}
\mathring{K}_5[H_1,H_2]=0 & \textup{in}\ [2M,R_1]\cup [2R_1,\infty),\\
\mathring{K}_5[H_1,H_2]\geq -\frac{30C_0}{(\log 2)r}|\nabla_L H_1|^2-\frac{3C_0}{(\log 2)r}|\nabla_L H_2|^2 & \textup{in}\ [R_1,2R_1].
\end{array}
\end{align*}
For the boundary term,
\begin{align*}
\begin{array}{cc}
e_5[H_1,H_2]= 3C_0 e_1[H_1,H_2] & \textup{in}\ [2M,R_1],\\
e_5[H_1,H_2]\geq e_1[H_1,H_2] & \textup{in}\ [R_1,\infty).
\end{array}
\end{align*}
\end{lemma}
Take $\epsilon_3>0$ such that
\begin{align*}
\begin{array}{cc}
\mathring{K}_3[H_1,H_2]\geq \epsilon_3\Big|\mathring{K}_2[H_1,H_2]\Big| & \textup{in}\ [r_0,r_1].
\end{array}
\end{align*}
Then the final current for $H_1$ and $H_2$ is defined by
\begin{align}
J^p[H_1,H_2]&:=\epsilon_3 J_2[H_1,H_2]+J_3[H_1,H_2]+\epsilon_2 J^p_4[H_1,H_2]+J_5[H_1,H_2],\\
K^{p}[H_1,H_2]&:=\Div J^p[H_1,H_2].
\end{align}
\begin{lemma}\label{H_12_final_current}
By the above construction, we have
\begin{align*}
K^p[H_1,H_2]\gtrsim &e^{p-1,deg}[H_1,H_2]\\
                    -&C\sqrt{e^{p-1,deg}[H_1,H_2]}\sqrt{e^{-1}[H_0,P,Q]}-Ce^{-1}[H_0,P,Q],
\end{align*}
\begin{align*}
\int_{\Sigma_\tau} J^p[H_1,H_2]\cdot n_{\Sigma_\tau} dVol_{\Sigma_\tau}\approx \int_{\Sigma_\tau} e^{p-1}[H_1,H_2] dVol_{\Sigma_\tau}.
\end{align*}
\end{lemma}
\begin{proof}
\begin{align*}
K^p[H_1,H_2]&=\mathring{K}^p[H_1,H_2]+Err^p[H_1,H_2],\\
\mathring{K}^p[H_1,H_2]&=\epsilon_3\mathring{K}_2[H_1,H_2]+\mathring{K}_3[H_1,H_2]+\epsilon_2 \mathring{K}^p_4[H_1,H_2]+\mathring{K}_5[H_1,H_2],\\
Err^p[H_1,H_2]&=\epsilon_3 Err_2[H_1,H_2]+Err_3[H_1,H_2]+\epsilon_2 Err^p_4[H_1,H_2]+Err_5[H_1,H_2].
\end{align*}
By the same argument as in Lemma \ref{H_0_final_current}, $\mathring{K}^p[H_1,H_2]\gtrsim e^{p-1,deg}[H_1,H_2]$. After using (\ref{def_P}) and (\ref{wavegauge}) to replace $\nabla_t H_1$ and $D^\dagger H_2$ in $Err_5[H_1,H_2]$, in any bounded $r$ we have
\begin{align*}
\Big|Err^p[H_1,H_2]\Big|^2\lesssim e^{-1,deg}[H_1,H_2] e^{-1}[H_0,P,Q]+e^{-1}[H_0,P,Q]^2.
\end{align*}
For large $r$, the only negative contribution comes from $K_3[H_1,H_2]$ and $K_5[H_1,H_2]$.
\begin{align*}
Err_{3,far}[H_1,H_2]&= -\frac{6M}{r^3}D^\dagger H_2\cdot \nabla_X H_1\\
            &= \frac{6M}{r^3}fD^\dagger H_2\cdot  \left(-\D^{-1}\nabla_T H_0+\frac{3}{r}H_1+\frac{1}{2r}D^\dagger H_2\right)\\
            &\geq -C(e^{-1}[H_1]+e^{-1}[H_2])-C\sqrt{e^{-1}[H_2]}\sqrt{e^{-1}[H_0]}.
\end{align*}
When $r$ is large enough, we can absorb $e^{-1}[H_1]+e^{-1}[H_2]$ by using $e^{p-1}[H_1]+e^{p-1}[H_2]$ in $\mathring{K}^p_4[H_1,H_2]$. 
\begin{align*}
Err_5[H_1,H_2]&=-\frac{6M}{r^3}D^\dagger H_2\cdot \nabla_T H_1\\
            &=-\frac{6M}{r^3}D^\dagger H_2\cdot\left( \frac{1}{r}\D P+\D\nabla_r H_0-\frac{1}{r}\D H_0 \right)\\
            &\geq -C\sqrt{e^{-1}[H_2]}\sqrt{e^{-1}[H_0,P]}.
\end{align*}
The estimate for the boundary term can be proved by the same way as in Lemma \ref{H_0_final_current}.
\end{proof}
By integrating $K^p[H_1,H_2]$ in $D(\tau_1,\tau_2)$ and using Cauchy-Schwarz inequality, we arrive at
\begin{proposition}
For any $\tau_2\geq \tau_1\geq 0$ and $p\in [\delta,2-\delta]$, we have
\begin{equation}\label{H_12 hierachy}
\begin{split}
&E^p[H_1,H_2](\tau_2)+\int_{\tau_1}^{\tau_2} E^{p-1,deg}[H_1,H_2](\tau)d\tau\\
\leq &C\left( E^p[H_1,H_2](\tau_1)+\int_{\tau_1}^{\tau_2} E^{-1}[H_0,P,Q](\tau)d\tau. \right).
\end{split}
\end{equation}
\end{proposition}
By combining (\ref{H_12 hierachy}) with Theorem \ref{thm_H0} and Corollary \ref{thm_PQ}, we can follow the proof of Corollary \ref{thm_PQ} to obtain the decay estimate for $H_1$ and $H_2$.
\begin{theorem}\label{thm_H12}
\begin{align}
E^p[H_1,H_2](\tau)\leq CI\tau^{-2+p+\delta_2},
\end{align}
where
\begin{align*}
I=\bar{E}[H_1,H_2]^{\leq 2}+\bar{E}[Q]^{\leq 5}+\bar{E}[H_0]^{\leq 5+m}+\bar{E}[P]^{\leq 8+m}.
\end{align*}
\end{theorem}
\section{Estimate for $\ell$=1}\label{ddd}
In this section we estimate the odd solution $h$ which is supported on the angular mode $\ell=1$. In this case $H_2$ vanishes automatically and (\ref{eqiation_H1H2}) becomes a single wave equation for $H_1$. Furthermore, it is well known \cite{Zerilli, Martel-Poisson} that in this mode, any solution of (\ref{LVEE}) is the linear combination of $K_{m}$, $m=-1,0,1$ and a pure gauge solution ${}^X\pi$. The components of $K_m$ is of the form
\begin{align*}
K_m=\frac{1}{r}Z^{1m}_\alpha dx^\alpha dt+\frac{2M}{r^2}\D^{-1}Z^{1m}_\alpha dx^\alpha dr.
\end{align*}
It's straightforward to verify that $K_m$ also satisfies the harmonic gauge condition (\ref{HG}). Therefore without loss of generality we can assume $h={}^X\pi$. One direct consequence is that $P$ vanishes and that (\ref{equation_H0}) becomes a wave equation without source. Thus by (\ref{H_0 hierarchy}), we have
\begin{theorem}\label{thm_l2}
For any $p\in [\delta,2-\delta]$,
\begin{align}
E^p[H_0](\tau)\leq C\bar{E}[H_0]^{\leq 2}\tau^{-2+p+\delta}.
\end{align}
\end{theorem}
However, we couldn't estimate $H_1$ directly from (\ref{eqiation_H1H2}). The reason is that $V_1+\frac{\ell(\ell+1)-1}{r^2}$ is negative on the horizon. Fortunately, the first derivatives of $H_1$ and $H_0$ are related by equations (\ref{wavegauge}) and (\ref{def_P}).
\begin{align*}
\frac{1}{r^3}\nabla_r\Big( r^3H_1 \Big)&=\D^{-1}\nabla_t H_0,\\
\frac{1}{r^3}\nabla_t\Big( r^3H_1 \Big)&=\D\left(\nabla_r H_0-\frac{1}{r}H_0\right).
\end{align*}
Hence $|\partial_\rho \Big( r^{3}H_1 \Big)|^2\leq C(r')e[H_0]$ for $r\in [2M,r']$. Together with the fact that $H_1=H_0$ on the horizon, we have for any $r_0\in [2M,r']$,
\begin{align*}
|r_0^{3} H_1(\tau,r_0)|&\leq |(2M)^{3} H_1(\tau,2M)|+\int_{2M}^{r_0} \left|\nabla_\rho \Big( r^{3}H_1 \Big)\right| d\rho\\
                        &\leq |(2M)^{3} H_0(\tau,2M)|+C(r')\int_{2M}^{r_0} \sqrt{e[H_0]} d\rho.
\end{align*}
Therefore,
\begin{align*}                        
|H_1(\tau,r_0)|^2&\leq C(r') \left(|H_0(\tau,2M)|^2+\int_{2M}^r e[H_0] d\rho\right).
\end{align*}
Integrating over $S^2$ and applying the mean value theorem to estimate $H_0(\tau,2M)$, we have
\begin{align}\label{FTC}
\int_{S^2(\tau,r_0)}|H_1|^2 dVol_{S^2}\leq C(r')\int_{\Sigma_\tau \cap\{r\leq r'\}} e[H_0] dVol_{\Sigma_\tau}.
\end{align}
Now we can define $J^p[H_1]$ as in section 4 with $H_2\equiv 0$. Both $K^p[H_1]$ and $J^p[H_1]\cdot n_{\Sigma_\tau}$ have negative zeroth order term in compact $r$ region. From (\ref{FTC}), we can add a large multiple of $J^{p}[H_0]^{\leq 1}$ to make the boundary and the bulk terms positive. In summary, we have
\begin{proposition}
For any $\tau_2>\tau_1\geq 0$ and $p\in [\delta,2-\delta],$ we have
\begin{equation}
\begin{split}
&E^p[H_1](\tau_2)+C_*E^{p}[H_0]^{\leq 1}(\tau_2)+\int_{\tau_1}^{\tau_2}E^{p-1,deg}[H_1](\tau)d\tau\\
\leq &C\Big(E^p[H_1](\tau_1)+C_*E^{p}[H_0]^{\leq 1}(\tau_1)\Big).
\end{split}
\end{equation}
\end{proposition}
\begin{theorem}\label{thm_l3}
For any $p\in [\delta,2-\delta]$,
\begin{align}
E^p[H_1](\tau)\leq C\left(\bar{E}[H_1]^{\leq 2}+\bar{E}[H_0]^{\leq 3}\right)\tau^{-2+p+\delta}.
\end{align}
\end{theorem}
Thus Theorem \ref{thm_main} follows form Theorem \ref{thm_H0}, \ref{thm_H12}, \ref{thm_l2} and \ref{thm_l3}.
\begin{appendix}
\section{derivation of equations}
In this appendix, we rewrite the Lichnerowicz d'Alembertian equation (\ref{LVEEG}) and the harmonic gauge condition (\ref{HG}) by using $H_0$, $H_1$, $H_2$ and the connections on $\mathcal{L}(-1)$ and $\mathcal{L}(-2)$. We first compute ${}^1\Box$ and ${}^2\Box$ in terms of ordinary derivatives.
\begin{lemma}
\begin{align}\label{box1}
{}^1\Box \Phi_\alpha dx^\alpha&= \left( -\D^{-1}\partial^2_t\Phi_\alpha +\D\partial^2_r\Phi_\alpha+\frac{2M}{r^2}\partial_r\Phi_\alpha+\frac{\mathring{\Delta}\Phi_\alpha}{r^2}-\frac{2M}{r^3}\Phi_\alpha  \right)dx^\alpha,
\end{align}
and
\begin{align}\label{box2}
{}^2\Box \Psi_{\alpha\beta} dx^\alpha dx^\beta= &\left( -\D^{-1}\partial^2_t\Psi_{\alpha\beta} +\D\partial^2_r\Psi_{\alpha\beta}-\frac{2}{r}\left( 1-\frac{3M}{r} \right)\partial_r\Psi_{\alpha\beta}\right.\\
&\left.+\frac{\mathring{\Delta}\Psi_{\alpha\beta}}{r^2}+\frac{2}{r^2}\left(1-\frac{4M}{r}\right)\Psi_{\alpha\beta}  \right)dx^\alpha.
\end{align}
\end{lemma}
\begin{proof}
From direct computation,
\begin{align*}
{}^1\nabla_A\Phi_\alpha &=\partial_A \Phi_\alpha-\Gamma_{A\alpha}^\beta \Phi_\beta\\
&=\tilde{\nabla}_A \Phi_\alpha-\frac{\tilde{\nabla}_A r}{r}\Phi_\alpha,\\
{}^1{\nabla}_B{}^1{\nabla}_A\Phi_\alpha&=\partial_B{}^1{\nabla}_A\Phi_\alpha-\Gamma_{BA}^C{}^1{\nabla}_C\Phi_\alpha-\Gamma_{BA}^\beta{}^1{\nabla}_\beta \Phi_\alpha -\Gamma_{B\alpha}^\beta \Phi_\beta\\
                          &=\tilde{\nabla}_B \left(\tilde{\nabla}_A \Phi_\alpha-\frac{\tilde{\nabla}_A r}{r}\Phi_\alpha \right) \frac{\tilde{\nabla}_B r}{r} \left(\tilde{\nabla}_A r \Phi_\alpha-\frac{\tilde{\nabla}_A r}{r} \Phi_\alpha \right)\\
                          &=\tilde{\nabla}_B\tilde{\nabla}_A r \Phi_\alpha-\frac{\tilde{\nabla}_B r}{r}\tilde{\nabla}_A \Phi_\alpha-\frac{\nabla_A}{r}\tilde{\nabla}_B \Phi_\alpha+\frac{1}{r^2}\Big(2\tilde{\nabla}_A r \tilde{\nabla}_B r-r\tilde{\nabla}_A\tilde{\nabla}_B r \Big)\Phi_\alpha,\\
{}^1{\nabla}_\beta \Phi_\alpha&=\partial_\beta \Phi_\alpha-\Gamma_{\beta\alpha}^\gamma \Phi_\gamma\\
                               &=\mathring{\nabla}_\beta \Phi_\alpha,\\
{}^1{\nabla}_\gamma{}^1{\nabla}_\beta \Phi_\alpha&=\partial_\gamma {}^1{\nabla}_\beta \Phi_\alpha-\Gamma_{\gamma\beta}^A{}^1{\nabla}_A \Phi_\alpha-\Gamma_{\gamma\beta}^\lambda{}^1{\nabla}_\lambda \Phi_\alpha-\Gamma_{\gamma\alpha}^\lambda {}^1{\nabla}_\beta \Phi_\lambda\\
                               &=\mathring{\nabla}_\gamma\mathring{\nabla}_\beta \Phi_\alpha+r \tilde{\nabla}^A r \tilde{\nabla}_A \Phi_\alpha\mathring{\sigma}_{\gamma\beta}-\tilde{\nabla}^A r \tilde{\nabla}_A r \Phi_\alpha\mathring{\sigma}_{\gamma\beta}.
\end{align*}
Therefore,
\begin{align*}
{}^1{\Box}\Phi_\alpha&=\tilde{\Box}\Phi_\alpha-\frac{2\tilde{\nabla}^A r}{r}\tilde{\nabla}_A\Phi_\alpha+\frac{1}{r^2}\left( 2\tilde{\nabla}^A r \tilde{\nabla}_A r-r\tilde{\Box}r \right)\Phi_\alpha\\
                      &+\frac{\mathring{\Delta}\Phi_\alpha}{r^2}+\frac{2\tilde{\nabla}^A r}{r}\tilde{\nabla}_A\Phi_\alpha-\frac{2\tilde{\nabla}^A r\tilde{\nabla}_A r}{r^2} \Phi_\alpha\\
                      &= \tilde{\Box}\Phi_\alpha+\frac{\mathring{\Delta}\Phi_\alpha}{r^2} -\frac{\tilde{\Box}r}{r} \Phi_\alpha.
\end{align*}
From $\tilde{\Box} f=-\left(1-\frac{2M}{r}\right)^{-1}\partial_t^2 f+ \left(1-\frac{2M}{r}\right)\partial_r^2 f +\frac{2M}{r^2}\partial_r f$ and $\mathring{\Delta}Y_\alpha=-\ell (\ell+1)+1$, (\ref{box1}) follows. Similarly,
\begin{align*}
{}^2{\nabla}_A \Psi_{\alpha\beta}&=\tilde{\nabla}_A\Psi_{\alpha\beta}-\frac{2\tilde{\nabla}_A r}{r} \Psi_{\alpha\beta},\\
{}^2{\nabla}_B{}^2{\nabla}_A\Psi_{\alpha\beta}&=\tilde{\nabla}_B\left(\tilde{\nabla}_A\Psi_{\alpha\beta}-\frac{2\tilde{\nabla}_A r}{r}\Psi_{\alpha\beta}\right)-\frac{2\tilde{\nabla}_B r}{r} \left(\tilde{\nabla}_A\Psi_{\alpha\beta}-\frac{2\tilde{\nabla}_A r}{r} \Psi_{\alpha\beta}\right)\\
 &=\tilde{\nabla}_B\tilde{\nabla}_A \Psi_{\alpha\beta}-\frac{2\tilde{\nabla}_A r}{r}\tilde{\nabla}_B\Psi_{\alpha\beta}-\frac{2\tilde{\nabla}_B r}{r}\tilde{\nabla}_A\Psi_{\alpha\beta}+\frac{1}{r^2}\Big( 6\tilde{\nabla}_A r\tilde{\nabla}_B r-2r\tilde{\nabla}_A\tilde{\nabla}_B r \Big)\Psi_{\alpha\beta},\\
{}^2{\nabla}_\gamma \Psi_{\alpha\beta}&=\mathring{\nabla}_\gamma \Psi_{\alpha\beta},\\
{}^2{\nabla}_\lambda{}^2{\nabla}_\gamma \Psi_{\alpha\beta}&=\mathring{\nabla}_\lambda\mathring{\nabla}_\gamma \Psi_{\alpha\beta}-\Gamma^A_{\lambda\gamma}{}^2{\nabla}_A \Psi_{\alpha\beta}\\
&=\mathring{\nabla}_\lambda\mathring{\nabla}_\gamma \Psi_{\alpha\beta}+r \tilde{\nabla}^A \tilde{\nabla}_A \Psi_{\alpha\beta}\mathring{\sigma}_{\lambda\gamma}-2\tilde{\nabla}^A\tilde{\nabla}_A \Psi_{\alpha\beta}\mathring{\sigma}_{\lambda\gamma}.
\end{align*}
\begin{align*}
{}^2{\Box}\Psi_{\alpha\beta}&=\tilde{\Box}\Psi_{\alpha\beta}-\frac{4\tilde{\nabla}^A r}{r}\tilde{\nabla}_A \Psi_{\alpha\beta}+\frac{1}{r^2}\Big( 6\tilde{\nabla}^A r\tilde{\nabla}_A r-2r\tilde{\Box}r \Big)\Psi_{\alpha\beta}\\
&+\frac{\mathring{\Delta}\Psi_{\alpha\beta}}{r^2}+\frac{2}{r}\tilde{\nabla}^A r\tilde{\nabla}_A \Psi_{\alpha\beta}-\frac{4}{r^2} \tilde{\nabla}^A r \tilde{\nabla}_A r \Psi_{\alpha\beta}\\
&=\tilde{\Box}\Psi_{\alpha\beta}-\frac{2\tilde{\nabla}^A r}{r}\tilde{\nabla}_A\Psi_{\alpha\beta}+\frac{\mathring{\Delta}\Psi_{\alpha\beta}}{r^2}\\
                             &+\frac{1}{r^2}\Big( 2 \tilde{\nabla}^A r \tilde{\nabla}_A r -2r\tilde{\Box}r \Big)\Psi_{\alpha\beta}.
\end{align*}
\end{proof}
Recall that for any odd tensor $h$, we define
\begin{align*}
H_0:=h_{t\alpha}dx^\alpha,\ H_1:=\left(1-\frac{2M}{r}\right)h_{r\alpha}dx^\alpha\in\ \Gamma(\mathcal{L}(-1)),
\end{align*} 
and
\begin{align*}
H_2=-h_{\alpha\beta}dx^\alpha dx^\beta \in \ \Gamma(\mathcal{L}(-2)).
\end{align*}
Also,
\begin{align*}
P&=P_{\alpha}dx^\alpha:=r^3\epsilon^{AB}\tilde{\nabla}_B (r^{-2}h_{A\alpha})dx^\alpha,\\
Q&=Q_{\alpha\beta}dx^\alpha dx^\beta:=\Big(\tilde{\nabla}^A r\Big)\left(\mathring{\nabla}_\alpha h_{A\beta}+\mathring{\nabla}_\beta h_{A\alpha}-r^2\tilde{\nabla}_A (r^{-2}h_{\alpha\beta})\right) dx^\alpha dx^\beta,\\
\tilde{Q}&=\tilde{Q}_{\alpha\beta}dx^\alpha dx^\beta:=T^A \left(\mathring{\nabla}_\alpha h_{A\beta}+\mathring{\nabla}_\beta h_{A\alpha}-r^2\tilde{\nabla}_A (r^{-2}h_{\alpha\beta})\right) dx^\alpha dx^\beta.
\end{align*}
In terms of ordinary derivatives, we have
\begin{align*}
P_\alpha&=r\D^{-1}\partial_t H_{1,\alpha}-r\partial_r H_{0,\alpha}+2H_{0,\alpha},\\
Q_{\alpha\beta}&=\mathring{\nabla}_\alpha H_{1,\beta}+\mathring{\nabla}_\beta H_{1,\alpha}+ \D\left(\partial_r H_{2,\alpha\beta}-\frac{2}{r}H_{2,\alpha\beta}\right),\\
\tilde{Q}_{\alpha\beta}&=\mathring{\nabla}_\alpha H_{0,\beta}+\mathring{\nabla}_\beta H_{0,\alpha}+ \partial_t H_{2,\alpha\beta},
\end{align*}
and the harmonic gauge condition is
\begin{align*}
0=-\D^{-1}\partial_t H_{0,\alpha}+\partial_r H_{1,\alpha}+\frac{2}{r}H_{1,\alpha}-\frac{1}{r^2}\mathring{\nabla}^\beta H_{2,\alpha\beta}.
\end{align*}
\begin{proof}[Proof of Lemma \ref{mainequation}]
In \cite{Hung-Keller-Wang}, $\delta Ric(h)_{r\alpha}=0$ can be rewritten as 
\begin{align*}
0=-\tilde{\nabla}_t \Big(r P_\alpha\Big)+\mathring{\nabla}^\beta Q_{\beta\alpha}.
\end{align*}
From this we obtain
\begin{align*}
0=&-r^2\D^{-1}\partial^2_t H_{1,\alpha}+r^2\partial_t\partial_r H_{0,\alpha}-2r\partial_t H_{0,\alpha}+\mathring{\Delta}H_{1,\alpha}\\
  &+\mathring{\nabla}^\beta\mathring{\nabla}_\alpha H_{1,\beta}+\D \partial_r\mathring{\nabla}^\beta H_{2,\alpha\beta}-\frac{2}{r}\D \mathring{\nabla}^\beta H_{2,\alpha\beta}.
\end{align*}
Using the harmonic gauge condition to replace $\partial_t H_{0,\alpha}$, we get
\begin{align*}
0=&-r^2\D^{-1}\partial^2_t H_{1,\alpha}+r^2\D \partial^2_r H_{1,\alpha}+2M\partial_r H_{1,\alpha}+\mathring{\Delta} H_{1,\alpha}\\
  &+\left(-5+\frac{16M}{r}\right)H_{1,\alpha}+\frac{2}{r}\left(1-\frac{3M}{r}\right)\mathring{\nabla}^\beta H_{2,\alpha\beta}\\
 =&r^2 \Box H_1+\left(-5+\frac{18M}{r}\right)  H_1-\left(1-\frac{3M}{r}\right) D^\dagger H_2.
\end{align*}
Similarly, $\delta Ric(h)_{t\alpha}=0$ becomes
\begin{align*}
0=-\D \tilde{\nabla}_r \Big(rP_\alpha\Big)+\mathring{\nabla}^\beta \tilde{Q}_{\beta\alpha}.
\end{align*}
Therefore
\begin{align*}
0=&-r^2\partial_r\partial_t H_{1,\alpha}+\left(-2r+2M\D^{-1}\right) \partial_t H_{1,\alpha}+r^2\D \partial_r^2 H_{0,\alpha}\\
  &-2\D H_{0,\alpha}+\mathring{\Delta}H_{0,\alpha}+H_{0,\alpha}+\partial_t\mathring{\nabla}^\beta H_{2,\alpha\beta}.
\end{align*}
Using the harmonic gauge condition to replace $\mathring{\nabla}^\beta H_{2,\alpha\beta}$, we have
\begin{align*}
0=&-r^2\D^{-1}\partial_t^2 H_{0,\alpha}+r^2\partial^2_r H_{0,\alpha}+\mathring{\Delta}H_{0,\alpha}+\left(-1+\frac{4M}{r}\right)H_{0,\alpha}+2M\D^{-1}\partial_t H_{1,\alpha}\\
  =&-r^2\D^{-1}\partial_t^2 H_{0,\alpha}+r^2\partial^2_r H_{0,\alpha}+\mathring{\Delta}H_{0,\alpha}+2M\partial_r H_{0,\alpha}-H_{0,\alpha}+\frac{2M}{r}P_\alpha\\
  =&r^2\Box H_0-\left(1-\frac{2M}{r}\right)H_0+\frac{2M}{r}P.
\end{align*}
The equation $\delta Ric(h)_{\alpha\beta}$ becomes
\begin{align*}
0&=-\D^{-1}\tilde{\nabla}_t \tilde{Q}_{\alpha\beta}+\tilde{\nabla}_r Q_{\alpha\beta}.
\end{align*}
Replacing $\partial_t H_{0,\alpha}$ again by the harmonic gauge condition, we get
\begin{align*}
0=&-\D^{-1}\partial_t^2 H_{2,\alpha\beta}+\D\partial_r^2 H_{2,\alpha\beta}-\frac{2}{r}\left(1-\frac{3M}{r}\right)\partial_r H_{2,\alpha\beta}\\
  &+\frac{\mathring{\Delta}H_{2,\alpha\beta}}{r^2}-\frac{8M}{r^3}H_{2,\alpha\beta}-\frac{2}{r}\Big( \mathring{\nabla}_\alpha H_{1,\beta}+\mathring{\nabla}_\beta H_{1,\alpha} \Big)\\
 =&\Box H_2-\frac{2}{r^2}H_2-\frac{2}{r^2}DH_1.
\end{align*}
\end{proof}
\section{Computation}
In this appendix, we include the computation of divergence of certain basic currents. Let $\phi$ be a smooth scalar function, $V$ be radial function, and $\Box\phi-V\phi=G$. Define the stress-energy tensor by
\begin{align*}
T_{ab}[\phi]=\nabla_a\phi\nabla_b\phi-\frac{1}{2}\left(\nabla^c\phi\nabla_c\phi\right)g_{ab}.
\end{align*}
The divergence of $T_{ab}$ is
\begin{align*}
\nabla^aT_{ab}[\phi]=\Box \phi\cdot\nabla_b \phi.
\end{align*}
Since $T$ is a Killing field,
\begin{equation}
\begin{split}
\nabla^a\left( \left( T_{ab}[\phi]-\frac{1}{2}V\phi^2 g_{ab} \right)T^b \right)=&\nabla_T\phi\left(\Box\phi-V\phi\right)\\
                                                                               =&\nabla_T\phi\cdot G.
\end{split}
\end{equation}
Consider the red-shift vector $Y$ in the $(v,R,\theta,\phi)$ coordinate.
\begin{align*}
Y\Big|_{r=2M}&=-2\frac{\partial}{\partial R},\\
\nabla_Y Y\Big|_{r=2M}&=-sT-\sigma Y,
\end{align*}
for some $s,\sigma>0$ to be determined. We have at $r=2M$ and under the $(v,R,\theta,\phi)$ coordinate,
\[ \nabla^{(a}Y^{b)}=\left[ \begin{array}{clclclc} \frac{s}{2} && -\frac{\sigma}{2} && 0 && 0 \\
                                          -\frac{\sigma}{2} && \frac{1}{2M} && 0 && 0 \\
                                          0 && 0 && -\frac{1}{4M^3} && 0 \\
                                          0 && 0 && 0 && -\frac{1}{4M^3\sin^2\theta}
                   \end{array} \right],\ \nabla_a Y^a=-\sigma-\frac{2}{M}.\]
The red-shift current $J_2[\phi]$ and density $K_2[\phi]$ are defined by
\begin{align}
J_2[\phi]_{a}&:=T_{ab} Y^b- \frac{1}{2} V|\phi|^2 Y_a,\\
K_2[\phi]&:=\Div J_2[\phi].
\end{align}
On the horizon we have
\begin{equation}\label{g_redshift}
\begin{split}
 K_2[\phi]&=T_{ab}[\phi]\nabla^aY^b+\Box\phi (\nabla_Y \phi)-V\phi (\nabla_Y \phi)-\frac{1}{2}(\nabla_Y V)\phi^2-\frac{1}{2}V\phi\nabla^aY_a\\
      &=\left(\frac{s}{2}|\nabla_v\phi|^2+\frac{1}{2M}|\nabla_R\phi|^2+\frac{2}{M}\nabla_R\phi\cdot\nabla_v\phi+\frac{\sigma}{2}|\slashed{\nabla}\phi|^2\right)\\
      &+\left( \frac{1}{2}\left(\sigma+\frac{2}{M}\right)V-\frac{1}{2}(\nabla_Y V) \right)\phi^2+(\nabla_Y\phi)\cdot G.
\end{split}      
\end{equation}
For any radial function $f(r)$, let $X=f(r)\left(1-\frac{2M}{r}\right)\frac{\partial}{\partial r}$. From direct computation,

\begin{align*}
 \nabla^{(a}\left(1-\frac{2M}{r}\right)\partial_r^{b)}&=\left[ \begin{array}{clclclc} -\frac{M}{r^2\left(1-\frac{2M}{r}\right)} && 0 && 0 && 0 \\
                                          0 && \frac{M}{r^2}\left(1-\frac{2M}{r}\right) && 0 && 0 \\
                                          0 && 0 && \frac{1}{r^3}\left(1-\frac{2M}{r}\right) && 0 \\
                                          0 && 0 && 0 && \frac{1}{r^3\sin^2\theta}\left(1-\frac{2M}{r}\right)
                   \end{array} \right],\\
                   \nabla_a \left(1-\frac{2M}{r}\right)\partial_r ^a&=\frac{2M}{r^2}+\frac{2}{r}\left(1-\frac{2M}{r}\right).\end{align*}

\begin{align*}
T_{ab}\nabla^a X^b=&\left(1-\frac{2M}{r}\right)^2\frac{\partial f}{\partial r}|\nabla_{r}\phi|^2+\frac{f}{r}\left(1-\frac{3M}{r}\right)|\slashed{\nabla}\phi|^2\\
&-\frac{1}{2}\left(1-\frac{2M}{r}\right)\left( \frac{\partial f}{\partial r}+\frac{2f}{r} \right)(\nabla \phi)^2.
\end{align*}
Let $\omega^X=\left(1-\frac{2M}{r}\right)\left(\frac{\partial f}{\partial r}+\frac{2f}{r}\right)$ and define the Morawetz current
\begin{align}
J_3[\phi]_a&:=T_{ab}X^b-\frac{1}{2}V|\phi|^2X_a+\frac{1}{4}\omega^X \nabla_a |\phi|^2-\frac{1}{4}\nabla_a \omega^X |\phi|^2,\\
K_3[\phi]&:=\Div J_3[\phi].
\end{align}
Then
\begin{align*}
K_3[\phi]=&T_{ab}[\phi]\nabla^a X^b+\Box\phi\cdot\nabla_X\phi-V\phi\cdot\nabla_X\phi-\frac{1}{2}\nabla_X V\phi^2-\frac{1}{2}\nabla^a X_a V\phi^2\\
&+\frac{1}{4}\omega^X\Big(2\phi\Box\phi+2(\nabla\phi^2)\Big)-\frac{1}{4}\Box\omega^X \phi^2\\
         =&\left(1-\frac{2M}{r}\right)^2\frac{\partial f}{\partial r}|\nabla_{r}\phi|^2+\frac{f}{r}\left(1-\frac{3M}{r}\right)|\slashed{\nabla}\phi|^2-\frac{1}{2}\left(1-\frac{2M}{r}\right)\left( \frac{\partial f}{\partial r}+\frac{2f}{r} \right)(\nabla \phi)^2\\
          &+\nabla_X\phi\cdot G-\frac{1}{2}\nabla_X V\phi^2-\frac{1}{2}\nabla^a X_a V\phi^2+\frac{1}{2}\omega^X(\nabla\phi)^2+\frac{1}{2}\omega^X\phi\cdot (V\phi+G)-\frac{1}{4}\Box \omega^X \phi^2\\
          =&\left(1-\frac{2M}{r}\right)^2\frac{\partial f}{\partial r}|\nabla_{r}\phi|^2+\frac{f}{r}\left(1-\frac{3M}{r}\right)|\slashed{\nabla}\phi|^2+\Big(\nabla_X\phi+\frac{1}{2}\omega^X\phi\Big)\cdot G\\
          &+\left(-\frac{1}{4}\Box\omega^X+\frac{1}{2}\omega^X V-\frac{1}{2}\nabla_X V-\frac{1}{2}\nabla^a X_a V\right)\phi^2.
\end{align*}
Together with $\nabla^aX_a=\omega^X+\frac{2M}{r^2}f$, we have
\begin{equation}\label{g_Morawetz}
\begin{split}
K_3[\phi]&=\mathring{K}_3[\phi]+Err_3[\phi],\\
\mathring{K}_3[\phi]=&\left(1-\frac{2M}{r}\right)^2\frac{\partial f}{\partial r}|\nabla_{r}\phi|^2+\frac{f}{r}\left(1-\frac{3M}{r}\right)|\slashed{\nabla}\phi|^2\\
&+\left(-\frac{1}{4}\Box\omega^X-\frac{1}{2}f\D \frac{\partial V}{\partial r}-\frac{M}{r^2}f V \right)|\phi|^2,\\
Err_3[\phi]&=\left(\nabla_X \phi+\frac{1}{2}\omega^X \phi\right)\cdot G.
\end{split}
\end{equation}
Let $\psi=r\phi$. By direct computation,
\begin{align*}
\Box \psi =\frac{2}{r}\D\psi_r+\frac{2M}{r^3}\psi+r\Box\phi.
\end{align*}
Let 
\begin{align*}
T_{ab}[\psi]= \nabla_a\psi\nabla_b\psi-\frac{1}{2}(\nabla\psi)^2g_{ab}.
\end{align*}
Then
\begin{align*}
\nabla^a T_{ab}[\psi]&=\Box \psi \nabla_b\psi\\
                     &=\frac{2}{r}\D\psi_r \nabla_b\psi+\frac{2M}{r^3}\psi \nabla_b\psi+r\Box\phi\nabla_b\psi.
\end{align*}
\begin{align*}
J^p_a =\frac{r^p}{r^2\D}T_{ab}[\psi]L^b-\frac{Mr^p}{r^5\D}\psi^2L_a.
\end{align*}
\begin{lemma}
\begin{equation}\label{g_rp}
\begin{split}
&\nabla^a \left(\frac{r^p}{r^2\D}T_{ab}[\psi]L^b-\frac{Mr^p}{r^5\D}\psi^2L_a-\frac{1}{2}\frac{r^p}{r^2\D}V\psi^2 L_a\right)\\
&=\D^{-2}r^{p-3}\left(\frac{p}{2}\D -\frac{M}{r}\right)|\nabla_L\psi|^2+\left(1-\frac{p}{2}\right)r^{p-3}|\nablas\psi|^2\\
&+\left((3-p)Mr^{p-6}-\frac{r^p}{r^2\D}\frac{1}{r}\left(1-\frac{M}{r}\right)V-\D\partial_r\left(\frac{r^p}{2r^2 \D}V\right)\right)|\psi|^2\\
&+\frac{r^p}{r^2\D}\nabla_L\psi \cdot rG.
\end{split}
\end{equation}
\end{lemma}
\begin{proof} For $p=0$,
\begin{align*}
&\nabla^a\left( \frac{1}{r^2\D}T_{ab}[\psi]L^b -\frac{M}{r^5\D}\psi^2 L_a\right)\\
                 =&I+II.
\end{align*}
\begin{align*}
I=&\frac{1}{r^2\D}\nabla^a T_{ab} L^b+\frac{\partial}{\partial r}\left(\frac{1}{r^2\D}\right)T_{ab} r^a L^b+\frac{1}{r^2\D} T_{ab}\nabla^{(a}L^{b)}\\
 =&I_1+I_2+I_3.
\end{align*}
\begin{align*}
I_1=&\frac{1}{r^2\D}\left(\frac{2}{r}\D\nabla_r \psi+\frac{2M}{r^3}\psi+r\Box\phi\right)\cdot\nabla_L\psi\\
   =&\frac{2}{r^3}\nabla_r\psi\nabla_L\psi+\frac{2M}{r^5\D}\psi\nabla_L\psi +\frac{1}{r^2\D}r\Box\phi\cdot\nabla_L\psi.
\end{align*}
\begin{align*}
I_2&=\frac{-2(r-M)}{r^4\D^2}\left( \D\nabla_r\psi \nabla_L\psi -\frac{1}{2}\D(\nabla\psi)^2\right)\\
   &=\frac{-2(r-M)}{r^4\D} \nabla_r\psi \nabla_L\psi +\frac{r-M}{r^4\D}(\nabla\psi)^2.
\end{align*}
\begin{align*}
I_3=\frac{1}{r^2\D}\left( -\frac{M}{r^2}\D^{-1}\nabla_L\psi\nabla_{\underline{L}}\psi+\frac{1}{r}\D|\nablas {\psi}|^2-\frac{r-M}{r^2}(\nabla\psi)^2 \right).
\end{align*}
In the computation of $I_3$, we used that 
\begin{align*}
\nabla^{(a}L^{b)}&=\frac{M}{r^2}\left(-\D^{-1}\partial^a_t\partial^b_t+\D\partial^a_r\partial^b_r\right)\\
                                &+\frac{1}{r^3}\D(\partial^a_\theta \partial^b_\theta+\sin\theta^{-2}\partial_\phi^a\partial_\phi^b ).
\end{align*}
Therefore,
\begin{align*}
I&=\frac{1}{r^3}|\nablas \psi|^2-\frac{2M}{r^4\D}\nabla_r\psi\nabla_L\psi+\frac{2M}{r^5\D}\psi\nabla_L\psi\\
 &-\frac{M}{r^4\D^2}\nabla_L\psi\nabla_{\underline{L}}\psi+\frac{1}{r^2\D}r\Box\phi\cdot\nabla_L\psi\\
 &=\frac{1}{r^3}|\nablas \psi|^2-\frac{M}{r^4\D^2}|\nabla_L\psi|^2+\frac{2M}{r^5\D}\psi\nabla_L\psi+\frac{1}{r^2\D}r\Box\phi\cdot\nabla_L\psi,
\end{align*}
\begin{align*}
II&=-\left[\nabla_L\left(\frac{M}{r^5\D}\right)\psi^2+\frac{2M}{r^5\D}\psi\nabla_L\psi+\frac{M}{r^5\D}\left(\frac{2r-2M}{r^2}\right)\psi^2\right]\\
  &=\frac{3M}{r^6}\psi^2-\frac{2M}{r^5\D}\psi\nabla_L\psi.
\end{align*}
Thus
\begin{align*}
I+II=\frac{1}{r^3}|\nablas \psi|^2+\frac{3M}{r^6}\psi^2-\frac{M}{r^4\D^2}|\nabla_L\psi|^2+\frac{1}{r^2\D}r\Box\phi\cdot\nabla_L\psi.
\end{align*}
For general $p$, we have
\begin{align*}
&\nabla^a\left( \frac{r^p}{r^2\D}T_{ab}[\psi]L^b -\frac{Mr^p}{r^5\D}\psi^2 L_a\right)\\
                  =&I+II. 
\end{align*}
\begin{align*}
I=r^{p-3}|\nablas \psi|^2+3Mr^{p-6}\psi^2-\frac{Mr^{p-4}}{\D^2}|\nabla_L\psi|^2+\frac{r^{p}}{r^2\D}r\Box\phi\nabla_L\psi.
\end{align*}
\begin{align*}
II&=\frac{pr^{p-1}}{r^2\D}\left( \D\nabla_r\psi \nabla_L\psi-\frac{1}{2}\D(\nabla\psi)^2 \right)-\frac{pMr^{p-1}}{r^5}\psi^2\\
  &=-\frac{p}{2}r^{p-3}|\nablas \psi|^2-pMr^{p-6}\psi^2+\frac{p}{2}r^{p-3}\D^{-1}|\nabla_L\psi|^2.
\end{align*}
Together with
\begin{align*}
&\nabla^a\left(\frac{r^p}{r^2\D}\left(-\frac{1}{2}V\psi^2 g_{ab}\right)L^b\right)\\
=&-\frac{r^p}{r^2\D} V\psi\nabla_L\psi\\
 &+\left(-\nabla_L\left(\frac{r^p}{2r^2\D} V\right)-\frac{r^p}{r^2\D}\frac{1}{r}\left(1-\frac{M}{r}\right)V\right)\psi^2,
\end{align*}
the result follows.
\end{proof}
\end{appendix}

\end{document}